\documentclass[11pt]{article}
\usepackage{amsfonts}
\usepackage{graphics}
\usepackage{indentfirst}
\usepackage{color}
\usepackage{cite}
\usepackage{latexsym}
\usepackage[paper=a4paper, left=1.6cm, right=1.6cm, top=1.8cm, bottom=1.6cm, headheight=5.5pt, footskip=0.8cm, footnotesep=0.8cm, centering, includefoot]{geometry}
\usepackage{amsmath}
\allowdisplaybreaks
\usepackage{amssymb}
\usepackage{esint}
\usepackage[colorlinks, linkcolor=red]{hyperref}
\hypersetup{urlcolor=red, citecolor=red}
\usepackage[dvips]{epsfig}
\usepackage{amscd}

\usepackage{amsthm}
\usepackage{mathrsfs}
\usepackage{verbatim}

\DeclareMathOperator{\divv}{div}
\DeclareMathOperator{\curl}{curl}

\DeclareMathOperator{\divf}
{di\overset{\raisebox{0.1ex}{\hspace{0.1em}$\mathbf{\cdot}$}}{v}}
\allowdisplaybreaks

\begin{document}
\title{Global weak solutions and incompressible limit to the isentropic compressible magnetohydrodynamic equations in 2D bounded domains with ripped density and large initial data
\thanks{
Wu's research was partially supported by Fujian Alliance of Mathematics (No. 2023SXLMMS08) and the Scientific Research Funds of Xiamen University of Technology (No. YKJ25009R).
Zhong's research was partially supported by Fundamental Research Funds for the Central Universities (No. SWU--KU24001) and National Natural Science Foundation of China (No. 12371227). }
}

\author{Shuai Wang$\,^{\rm 1}\,$,\ Guochun Wu$\,^{\rm 2}\,$,\
Xin Zhong$\,^{\rm 1}\,$ {\thanks{E-mail addresses: swang238@163.com (S. Wang),
guochunwu@126.com (G. Wu), xzhong1014@amss.ac.cn (X. Zhong).}}\date{}\\
\footnotesize $^{\rm 1}\,$
School of Mathematics and Statistics, Southwest University, Chongqing 400715, P. R. China\\
\footnotesize $^{\rm 2}\,$ School of Mathematics and Statistics, Xiamen University of Technology, Xiamen 361024, P. R. China}

\maketitle
\newtheorem{theorem}{Theorem}[section]
\newtheorem{definition}{Definition}[section]
\newtheorem{lemma}{Lemma}[section]
\newtheorem{proposition}{Proposition}[section]
\newtheorem{corollary}{Corollary}[section]
\newtheorem{remark}{Remark}[section]
\renewcommand{\theequation}{\thesection.\arabic{equation}}
\catcode`@=11 \@addtoreset{equation}{section} \catcode`@=12
\maketitle{}

\begin{abstract}
In our previous work (arXiv:2510.00812), we have shown the global existence and incompressible limit of weak solutions to the isentropic compressible magnetohydrodynamic equations involving ripped density and large initial energy in the whole plane. In this paper we generalize such results to the case of two-dimensional bounded convex domains under Navier-slip boundary conditions. When comparing to the known results for
global solutions of the initial-boundary value problem, we obtain uniform {\it a priori} estimates independent of the bulk viscosity coefficient.
\end{abstract}




\tableofcontents

\section{Introduction}
\subsection{Background and motivation}
We continue to study global weak solutions and incompressible limit of two-dimensional (2D) isentropic compressible magnetohydrodynamic equations with discontinuous initial data and vacuum when the bulk viscosity is suitably large.
Let $\Omega\subset\mathbb{R}^2$ be a bounded convex domain with smooth boundary $\partial\Omega$, the evolution of a conducting fluid under the effect of a electromagnetic field occupying the spatial domain $\Omega$ is governed by the
isentropic compressible magnetohydrodynamic equations
\begin{align}\label{a1}
\begin{cases}
\rho_t+\divv(\rho\mathbf{u})=0,\\
(\rho\mathbf{u})_t+\divv(\rho\mathbf{u}\otimes\mathbf{u})+\nabla P=
\mu\Delta\mathbf{u}+(\mu+\lambda)\nabla\divv\mathbf{u}+\mathbf{B}\cdot\nabla\mathbf{B}-\frac12\nabla|\mathbf{B}|^2,\\
\mathbf{B}_t+\mathbf{u}\cdot\nabla\mathbf{B}-\mathbf{B}\cdot\nabla\mathbf{u}+\mathbf{B}\divv\mathbf{u}=\nu\Delta\mathbf{B},\\
\divv\mathbf{B}=0
\end{cases}
\end{align}
with the given initial data
\begin{equation}
(\rho,\mathbf{u},\mathbf{B})|_{t=0}=(\rho_0,\mathbf{u}_0,\mathbf{B}_0)(\mathbf{x}),\ \ \mathbf{x}\in\Omega,
\end{equation}
and Navier-slip boundary conditions
\begin{equation}\label{a3}
\mathbf{u}\cdot \mathbf{n}=\curl\mathbf{u}=0,\ \mathbf{B}\cdot \mathbf{n}=\curl\mathbf{B}=0,\ \mathbf{x}\in\partial\Omega,\ t>0,
\end{equation}
where $\mathbf{n}=(n^1,n^2)$ is the unit outward normal vector to $\partial\Omega$. Here the unknowns $\rho$, $\mathbf{u}=(u^1,u^2)$, $\mathbf{B}=(B^1,B^2)$, and $P=P(\rho)=a\rho^\gamma\ (a>0,\gamma>1)$ represent the density, velocity, magnetic field, and pressure, respectively. The constants $\mu$ and $\lambda$ represent shear viscosity and bulk viscosity, respectively, satisfying the physical restrictions
\begin{equation*}
\mu>0,\ \ \ \mu+\lambda\geq0,
\end{equation*}
while $\nu>0$ is the resistivity coefficient.

Upon integrating by parts, we find that system \eqref{a1} verifies the global energy law
\begin{align*}
\int_{\Omega}\bigg(\frac{1}{2}\rho |\mathbf{u}|^2+\frac{1}{2} |\mathbf{B}|^2+G(\rho)\bigg)\mathrm{d}\mathbf{x}
+\int_0^t\int_{\Omega}\big[(2\mu+\lambda)(\divv \mathbf{u})^2+\mu(\curl\mathbf{u})^2+\nu(\curl \mathbf{B})^2\big]\mathrm{d}\mathbf{x}\mathrm{d}\tau
\leq C_0,
\end{align*}
with the initial total energy $C_0$ and the potential energy $G(\rho)$ being given by
\begin{equation}\label{1.4}
C_0\triangleq\int_{\Omega}
\Big(\frac{1}{2}\rho_0|\mathbf{u}_0|^2+\frac{1}{2}|\mathbf{B}_0|^2+G(\rho_0)\Big)\mathrm{d}\mathbf{x},~~
G(\rho)\triangleq\rho\int_{\bar{\rho}}^\rho\frac{P(\xi)-P(\bar{\rho})}{\xi^2}\mathrm{d}\xi,
\end{equation}
where the average of density over $\Omega$, due to the mass equation \eqref{a1}$_1$, is a positive constant
\begin{equation*}
  \bar{\rho}\triangleq\frac{1}{|\Omega|}\int_\Omega \rho\mathrm{d}\mathbf{x}=\frac{1}{|\Omega|}\int_\Omega \rho_0\mathrm{d}\mathbf{x}.
\end{equation*}
Asymptotically, the incompressibility is encoded in the global energy law via the bulk viscosity limit $\lambda\rightarrow\infty$.

Magnetohydrodynamics (MHD) provides a continuum framework for studying the interplay between fluid dynamics and electromagnetism in conducting media, with broad applications in astrophysics, geophysics, and plasma physics (see, e.g., \cite{DA17}). The core mathematical challenge, however, stems from the strongly coupled and nonlinear nature of the governing equations, where hydrodynamic transport and magnetic induction mutually constrain each other. This coupling renders the establishment of global well-posedness exceptionally difficult, yet a rigorous mathematical foundation is essential for reliable physical prediction. Let us point out that system \eqref{a1} can be derived by combining the Navier--Stokes equations for isentropic compressible flow with Maxwell's equations in free space and Ohm's law, and we refer the reader to \cite[Chapter 3]{LT12} for a detailed derivation.

There has been very exciting research on the multi-dimensional isentropic compressible MHD equations and important progress has been made on the global well-posedness over the past four decades. The first global result is due to Kawashima \cite{K84} where he proves the existence of global-in-time classical solutions in 2D if the initial data are sufficiently close to the constant state in $H^3$. Subsequently, Li--Yu \cite{LY11} and Zhang--Zhao \cite{ZZ10} independently demonstrated the global existence and uniqueness of classical solutions for initial data that are small perturbations in $H^3(\mathbb{R}^3)$ around a non-vacuum equilibrium state (constant). Moreover, it is also interesting to deal with the presence of vacuum where new phenomena will appear because the momentum equation becomes singular and degenerate near the vacuum region. In this case, the 3D and 2D Cauchy problem were shown to admit global strong solutions in \cite{LXZ13} and \cite{LSX16}, respectively, with initial data having small energy but possibly large oscillations. One may ask if we can weaken the small-energy assumption. The global well-posedness of classical solutions for \eqref{a1} in $\mathbb{R}^3$ derived in \cite{HHPZ17} ensures this possibility since only the quantity $\big[(\gamma-1)^{1/9}+\nu^{-1/4}\big]C_0$ is assumed to be suitably small. In addition to the Cauchy problem, the initial-boundary value problem is equally important physically and mathematically. Recently, the result in \cite{LXZ13} was extended to the cases of 3D bounded domains and exterior domains with Navier-slip boundary conditions in \cite{CHPS23,CHS24}. Particularly, these results reveal that different types of boundaries require generally distinct approaches.

Another category concerning the solvability of \eqref{a1} is that of
the so-called finite-energy weak solutions: solutions in the sense of distributions satisfying the energy inequality for which one can guarantee their global existence for arbitrary large initial data.
Based on classical weak convergence methods developed by P.-L. Lions \cite{PL98} and Feireisl \cite{F04}, Hu and Wang \cite{HW10} proved the global existence and large-time behavior of finite-energy weak solutions in 3D bounded domains under Dirichlet boundary conditions when the adiabatic exponent $\gamma>\frac32$. These weak solutions provide valuable insights into the system's behavior in extreme regimes although few things are known regarding the uniqueness of such solutions.

A third category of results has to do with an intermediate regularity functional framework which was pioneered in the works of Hoff \cite{Hoff95,Hoff95*,Hoff02,HS08} (which we will call Hoff's solutions).
By intermediate regularity we mean that these solutions possess stronger regularity than finite-energy weak solutions in the sense that particle paths can be defined in non-vacuum regions, yet weaker than strong solutions as they may have discontinuous density along some curves (in 2D) or surfaces (in 3D) (see \cite{Hoff02}). From this perspective, Suen and Hoff \cite{SH12} established the global existence of Hoff's solutions to \eqref{a1} provided the gradients of initial velocity and magnetic field are bounded in $L^2(\mathbb{R}^3)$ and the initial density is positive and essentially bounded. Such a momentous result was soon extended to the presence of initial vacuum in \cite{LYZ13}. Meanwhile, with improved regularities for initial data, the uniqueness and continuous dependence of Hoff's solutions as well as optimal time-decay rates were obtained in \cite{S20} and \cite{WZZ21}, respectively, providing a crucial generalization of the existence theory in \cite{SH12}. It should be emphasized that all results in \cite{SH12,S20,LYZ13,WZZ21} require small initial energy. However, these solutions are interesting as they allow us to work with discontinuous density while granting some extra regularity for the velocity field.

More recently, the authors of the present paper studied the Cauchy problem of \eqref{a1} in $\mathbb{R}^2$, and showed the global existence and incompressible limit of Hoff's solutions in \cite{WWZ3} for large bulk viscosity. {\it A natural question then appears: is it possible to establish the global-in-time existence and incompressible limit of Hoff's solutions to the 2D initial-boundary value problem \eqref{a1}--\eqref{a3}? The main contribution of this paper is to show that this is indeed the case}.
Precisely, we would like to address the problem \eqref{a1}--\eqref{a3} supplemented with general arbitrary large initial energy involving merely nonnegative bounded density. We will prove the global existence of Hoff's solutions to \eqref{a1}--\eqref{a3} with large initial data as long as the bulk viscosity coefficient is properly large and investigate the limiting behavior of such solutions as $\lambda\rightarrow\infty$. As mentioned above, the recovery of incompressibility in the limit leads us to expect a global solution to the inhomogeneous incompressible magnetohydrodynamic system in $\Omega\times(0,+\infty)$:
\begin{align}\label{1.5}
\begin{cases}
\varrho_t+{\bf v}\cdot\nabla\varrho=0,\\
\varrho{\bf v}_t+\varrho{\bf v}\cdot\nabla {\bf v}+\nabla \Pi=\mu\Delta {\bf v}+{\bf b}\cdot\nabla{\bf b}-\frac12\nabla|{\bf b}|^2,\\
{\bf b}_t+\mathbf{v}\cdot\nabla{\bf b}={\bf b}\cdot\nabla\mathbf{v}+\nu\Delta{\bf b},\\
\divv{\bf v}=\divv{\bf b}=0,\\
(\varrho,{\bf v},{\bf b})|_{t=0}=(\rho_0,{\bf v}_0,{\bf B}_0),
\end{cases}
\end{align}
where ${\bf v}_0$ is the Leray-Helmholtz projection of ${\bf u}_0$ on divergence-free vector fields.

\subsection{Main results}
Before stating our main results, we first formulate the notations and conventions adopted herein. We denote by $C$ a generic positive constant which may vary at different places. The symbol $\Box$ marks the end of a proof, $A:B$ represents the trace of the matrix product $AB^\top$, and $c\triangleq d$ means $c=d$ by definition. For notational simplicity, we write
\begin{align*}
\int f \mathrm{d}\mathbf{x}=\int_{\Omega} f \mathrm{d}\mathbf{x},~~ f_i=\partial_if\triangleq\frac{\partial f}{\partial x_i},~~
\bar{f}\triangleq\fint f\mathrm{d}\mathbf{x}=\frac{1}{|\Omega|}\int_\Omega f\mathrm{d}\mathbf{x}.
\end{align*}
For $1\le p\le \infty$ and integer $k\ge 0$, we denote the standard Sobolev spaces
\begin{align*}
L^p=L^p(\Omega),\ \ W^{k, p}=W^{k, p}(\Omega),\ \ H^k=W^{k, 2},\ \
H_\omega^k=\{\mathbf{f}\in H^k :\,(\mathbf{f}\cdot \mathbf{n})|_{\partial\Omega}=\curl \mathbf{f}|_{\partial\Omega}=0\}.
\end{align*}
Moreover, for $\alpha\in (0,1]$, the H\" older seminorm of a function ${\bf v}:U\subseteq\overline{\Omega}\rightarrow \mathbb R^2$ is defined by
\begin{align*}
\langle {\bf v}\rangle^\alpha_{U}
=\sup\limits_{\substack{{\bf x}, {\bf y}\in U\\ {\bf x}\neq {\bf y}}}
\frac{|{\bf v}({\bf x})-{\bf v}({\bf y})|}{|{\bf x}-{\bf y}|^\alpha}.
\end{align*}

The analysis relies on three key quantities
\begin{align}\label{1.7}
\begin{cases}
\dot{f}\triangleq f_t+\mathbf{u}\cdot\nabla f,\\
F\triangleq(2\mu+\lambda)\divv\mathbf{u}-(P(\rho)-\bar{P})-\frac12|\mathbf{B}|^2,\\
\omega\triangleq\curl\mathbf{u}=-\nabla^\bot\cdot\mathbf{u}=-\partial_2u^1+\partial_1u^2,
\end{cases}
\end{align}
which represent the material derivative of $f$, the effective viscous flux, and the vorticity, respectively.
In addition, we introduce the Leray-Helmholtz projector
\begin{equation*}
\mathcal{P}\triangleq \text{Id}+\nabla(-\Delta)^{-1}\divv
\end{equation*}
which projects onto the subspace of divergence-free vector fields inheriting the no-penetration boundary condition, along with its complement $\mathcal{Q}\triangleq \text{Id}-\mathcal{P}$. Both $\mathcal{P}$ and $\mathcal{Q}$ are bounded operators on $L^p$ for any $1<p<\infty$.

We recall the weak solutions to the initial-boundary value problem \eqref{a1}--\eqref{a3} as defined in \cite{SH12}.
\begin{definition}\label{d1.1}
A triplet $(\rho, \mathbf{u}, \mathbf{B})$ is said to be a weak solution to the problem \eqref{a1}--\eqref{a3} provided that
\begin{equation*}
\rho\in C([0,\infty);H^{-1}(\Omega)),\ \
(\rho\mathbf{u},\mathbf{B})\in C([0,\infty);\widetilde{H}^{1}(\Omega)^*),\ \
(\nabla\mathbf{u}, \nabla\mathbf{B})\in L^2(\Omega\times(0,\infty))
\end{equation*}
with $\divv \mathbf{B}(\cdot,t)=0$ in $\mathcal{D}'(\Omega)$ for $t>0$ and  $(\rho,\mathbf{u},\mathbf{B})|_{t=0}=(\rho_0,\mathbf{u}_0,\mathbf{B}_0)$,
where $\widetilde{H}^{1}(\Omega)^*$ is the dual space of $\widetilde{H}^{1}(\Omega)=\{\mathbf{f}\in H^1:(\mathbf{f}\cdot \mathbf{n})|_{\partial\Omega}=0\}$.
Moreover, for any $t_2\ge t_1\ge 0$ and any test function $(\phi,\boldsymbol\psi)({\bf x},t)\in C^1(\overline{\Omega}\times[t_1,t_2])$, with uniformly bounded support in ${\bf x}$ for $t\in[t_1,t_2]$ and satisfying $(\boldsymbol\psi\cdot\mathbf{n})|_{\partial\Omega}=0$, the following identities hold\footnote{Throughout this paper, we will use the Einstein summation over repeated indices convention.}:
\begin{align*}
&\int\rho(\mathbf{x},\cdot)\phi(\mathbf{x},\cdot)
\mathrm{d}\mathbf{x}\Big|_{t_1}^{t_2}=\int_{t_1}^{t_2}\int(\rho\phi_t+
\rho\mathbf{u}\cdot\nabla\phi)\mathrm{d}\mathbf{x}\mathrm{d}t,\\
&\int_{\Omega}(\rho \mathbf{u}\cdot\boldsymbol\psi)
(\mathbf{x},\cdot)\mathrm{d}\mathbf{x}\Big|_{t_1}^{t_2}
+\int_{t_1}^{t_2}\int\big[(2\mu+\lambda)\divv\mathbf{u}\divv\boldsymbol{\psi}+\mu\curl\mathbf{u}\curl\boldsymbol{\psi}\big]\mathrm{d}\mathbf{x}\mathrm{d}t\notag\\
&\qquad=\int_{t_1}^{t_2}
\int\Big[\rho\mathbf{u}\cdot\boldsymbol{\psi}_t+(\rho\mathbf{u}\otimes\mathbf{u}):\nabla\boldsymbol{\psi}+P\divv\boldsymbol{\psi}
-(\mathbf{B}\otimes\mathbf{B}):\nabla\boldsymbol{\psi}
+\frac{1}{2}|\mathbf{B}|^2\divv\boldsymbol{\psi}
\Big]\mathrm{d}\mathbf{x}\mathrm{d}t,\\
&\int(\mathbf{B}\cdot\boldsymbol\psi)(\mathbf{x},\cdot)
\mathrm{d}\mathbf{x}\Big|_{t_1}^{t_2}=\int_{t_1}^{t_2}\int\big[\mathbf{B}\cdot\boldsymbol{\psi}_t
+(\mathbf{u}\otimes\mathbf{B}):\nabla\boldsymbol{\psi}-(\mathbf{B}\otimes\mathbf{u}):\nabla\boldsymbol{\psi}
-\nu\curl\mathbf{B}\curl\boldsymbol{\psi}\big]\mathrm{d}\mathbf{x}\mathrm{d}t.
\end{align*}
\end{definition}
For the initial data $(\rho_0,\mathbf{u}_0,\mathbf{B}_0)$, we always assume that there exist two positive constants $\hat{\rho}$ and $M$ (not necessarily small) satisfying
\begin{gather}\label{c1}
0\leq\inf\rho_0\leq\sup\rho_0\leq\hat{\rho},\ \ (\mathbf{u}_0,\mathbf{B}_0)\in H_\omega^{1},\ \ \divv \mathbf{B}_0=0,
\\
C_0+(2\mu+\lambda)\|\divv\mathbf{u}_0\|_{L^2}^2+\mu\|\curl\mathbf{u}_0\|_{L^2}^2+\nu\|\curl\mathbf{B}_0\|_{L^2}^2
\leq M.\label{c2}
\end{gather}

We now state our first result on the global existence of weak solutions.

\begin{theorem}\label{t1.1}
Let \eqref{c1} and \eqref{c2} be satisfied, there exists a positive number $D$ depending only on $\Omega$, $\hat{\rho}$, $a$, $\gamma$, $\nu$, and $\mu$ such that if
\begin{equation}\label{lam}
\lambda\geq\exp\bigg\{(2+M)^{e^{D(1+C_0)^3}}\bigg\},
\end{equation}
then the initial-boundary value problem \eqref{a1}--\eqref{a3} admits a global weak solution $(\rho,\mathbf{u},\mathbf{B})$ in the sense of Definition $\ref{d1.1}$ satisfying
\begin{equation}\label{reg}
\begin{cases}0\leq\rho(\mathbf{x},t)\leq2\hat{\rho}~a.e.~\mathrm{on}~\Omega\times[0,\infty),\\
(\rho,\sqrt{\rho}\mathbf{u},\mathbf{B})\in C([0,\infty);L^2(\Omega)),~(\mathbf{u},\mathbf{B})\in L^\infty(0,\infty;H^1(\Omega)),\\
(\nabla^2\mathcal{P}\mathbf{u},\nabla F,\sqrt{\rho}\dot{\mathbf{u}},\mathbf{B}_t,\nabla\curl\mathbf{B})\in L^2(\Omega\times(0,\infty)),\\
\sigma^{\frac{1}{2}}(\sqrt{\rho}\dot{\mathbf{u}},\mathbf{B}_t,\nabla\curl\mathbf{B})\in L^\infty(0,\infty;L^2(\Omega)),~ \sigma^{\frac{1}{2}}(\nabla\dot{\mathbf{u}},\curl\mathbf{B}_t)
\in L^2(\Omega\times(0,\infty)),
\end{cases}
\end{equation}
where $\sigma\triangleq\min\{1,t\}$.
\end{theorem}

The next result will treat the incompressible limit (characterised by the large value of the bulk viscosity) of the global weak solutions established in Theorem \ref{t1.1}.

\begin{theorem}\label{t1.2}
Let $\{(\rho^\lambda,{\bf u}^\lambda,{\bf B}^\lambda)({\bf x},t)\}$ be the family of solutions from Theorem \ref{t1.1}. Then, there exists a subsequence $\{\lambda_k\}$
with $\lambda_k\rightarrow\infty$ such that
\begin{align}\label{1.11}
 &\rho^{\lambda_{k}}\rightarrow \varrho ~~\text{strongly in} ~ L^2(K), \ \
 \text{for any compact set} \ K \subset \Omega \text{ and any} \ t\ge 0 ,\\
 &{\bf u}^{\lambda_{k}}\rightarrow {\bf v},~~{\bf B}^{\lambda_{k}}\rightarrow {\bf b}
 ~~\text{uniformly on compact sets in}~ \Omega\times(0,\infty),\notag
\end{align}
where $(\varrho,{\bf v},{\bf b})$ is a global weak solution to the inhomogeneous incompressible MHD equations \eqref{1.5} in the sense of Definition \ref{d1.2} below.
\end{theorem}

\begin{definition}\label{d1.2}
A triplet $(\varrho, {\bf v}, {\bf b})$ is said to be a weak solution to the problem \eqref{1.5} provided that
\begin{gather}
\varrho\in L^\infty(\Omega\times (0,\infty)),~~
(\sqrt{\varrho}\mathbf{v},\mathbf{b})\in L^\infty([0,\infty); L^2(\Omega)),~~
({\bf v}, \nabla{\bf v},\nabla\mathbf{b})\in L^2(\Omega\times (0,\infty)),\notag\\
\varrho\in C([0,\infty);L^{2}(\Omega)).\label{1.12}
\end{gather}
Moreover, for any $t_2\geq t_1\geq 0$ and any $C^1$ test function $(\phi,\boldsymbol\psi)$ just as in Definition \ref{d1.1},
which additionally satisfies $\divv\boldsymbol\psi(\cdot,t)=0$ on $\Omega\times[0,\infty)$, the following identities hold:
\begin{align}\label{1.13}
\int(\varrho\phi)({\bf x},\cdot)\mathrm{d}{\bf x}\Big |_{t_1}^{t_2}&=\int_{t_1}^{t_2}\int(\varrho\phi_t+\varrho{\bf v}\cdot\nabla\phi)\mathrm{d}{\bf x}\mathrm{d}t,
\\
\int(\varrho{\bf v}\cdot\boldsymbol\psi)({\bf x},\cdot)\mathrm{d}{\bf x}\Big |_{t_1}^{t_2}
&=
\int_{t_1}^{t_2}\int\big[\varrho{\bf v}\cdot\boldsymbol\psi_t
+(\varrho\mathbf{v}\otimes\mathbf{v}):\nabla\boldsymbol{\psi}
-(\mathbf{b}\otimes\mathbf{b}):\nabla\boldsymbol{\psi}-\mu\curl\mathbf{v}\curl\boldsymbol{\psi}\big]
\mathrm{d}{\bf x}\mathrm{d}t,\label{1.14}
\\
\int({\bf b}\cdot\boldsymbol\psi)({\bf x},\cdot)
\mathrm{d}\mathbf{x}\Big|_{t_1}^{t_2}&=\int_{t_1}^{t_2}\int
\big[{\bf b}\cdot\boldsymbol\psi_t+(\mathbf{v}\otimes\mathbf{b}):\nabla\boldsymbol{\psi}
-(\mathbf{b}\otimes\mathbf{v}):\nabla\boldsymbol{\psi}-\nu\curl\mathbf{b}\curl\boldsymbol{\psi}\big]
\mathrm{d}\mathbf{x}\mathrm{d}t.\label{1.15}
\end{align}
\end{definition}

Several remarks are in order.

\begin{remark}
It should be noted that Theorem \ref{t1.1} not only holds for arbitrarily large initial energy as long as the bulk viscosity coefficient is suitably large, which is in sharp contrast to \cite{CHPS23,CHS24,LXZ13,SH12,LYZ13,LSX16}
where the smallness condition on the initial energy is needed, but also features an
estimation constant independent of the bulk viscosity coefficient.
\end{remark}

\begin{remark}
Theorems \ref{t1.1} and \ref{t1.2} extend our previous work on the whole plane case \cite{WWZ3} to the bounded convex domains. However, this is a non-trivial generalization because boundary will bring out some new difficulties (see subsection \ref{ssec3}). Naturally, one may ask if analogous results hold for 2D domains with distinct boundary conditions, such as Dirichlet or free boundary conditions. New ideas are required to tackle these cases, which will be left for future studies.
\end{remark}

\begin{remark}
Compared with our previous results regarding the isentropic compressible Navier--Stokes equations (i.e., $\mathbf{B}\equiv\mathbf{B}_0\equiv\mathbf{0}$ in \eqref{a1}) \cite{WWZ2}, the presence of the magnetic field subject to boundaries poses significant challenges. First, the crucial Poincar\'e inequality is no longer directly applicable to the effective viscous flux; secondly, one needs stronger viscous dissipation to counteract magnetic nonlinearity reflected in the assumption \eqref{lam}.
\end{remark}

\subsection{Strategy of the proof}\label{ssec3}

We now present a roadmap of our approach and the key challenges in the proof.
Our strategy unfolds in two principal steps. The global smooth approximate solutions are constructed by leveraging the local existence theory for initial data with positive density, supplemented by a blow-up criterion (Lemma \ref{l2.1}). This, in turn, paves the way for proving Theorem \ref{t1.1} via compactness arguments. Subsequently, based on the uniform estimates secured in the proof of Theorem \ref{t1.1}, we derive Theorem \ref{t1.2} by a mollification argument. Therefore, the core of our analysis is to establish uniform {\it a priori} estimates independent of the lower bound of density $\rho$ and the bulk viscosity $\lambda$. However, in comparison with our previous work concerning the 2D Cauchy problem \cite{WWZ3}, one of major difficulties lies in dealing with many surface integrals caused by the boundary condition \eqref{a3}. Moreover, it should be emphasized that crucial techniques in the case of 3D initial-boundary value problem \cite{CHPS23} cannot be adopted to the situation treated here, since their arguments are only valid for the small initial energy. Consequently, some new ideas are needed to overcome these obstacles.

It was shown (with minor modification) in \cite{DW15} that if $0<T^*<\infty$ is the maximal existence time of strong solutions to \eqref{a1}--\eqref{a3}, then
\begin{align*}
\limsup\limits_{T\nearrow T^*}\|\rho\|_{L^\infty(0,T;L^\infty)}=\infty,
\end{align*}
which implies that the key uniform-in-$\lambda$ {\it a priori} estimate is to obtain the time-independent upper bound of the density. We begin with the $L^\infty(0, T; L^2)$-norm of the gradient of velocity (see Lemma \ref{l3.2}), from which the key process is to handle the term $\int\rho \dot{\mathbf{u}}\cdot(\mathbf{u}\cdot\nabla)\mathbf{u}\mathrm{d}\mathbf{x}$ in \eqref{3.10} after multiplying \eqref{a1}$_2$ by ${\bf u}_t$.
Based on an important Desjardins-type logarithmic interpolation inequality (see Lemma $\ref{log}$) on the two-dimensional bounded convex domains,
this in turn can be controlled by $\frac{1}{(2\mu+\lambda)^2}\|P(\rho)-\bar{P}\|_{L^4}^4$ (see \eqref{3.9}--\eqref{3.10}).
Hence, the crucial issue is the integrability in time of $\frac{1}{(2\mu+\lambda)^2}\|P(\rho)-\bar{P}\|_{L^4}^4$ when applying Gronwall's inequality, which leads us to assume the {\it a priori hypothesis} \eqref{3.1} in Proposition \ref{p3.1}. Then, the next key ingredient is to complete the proof of the {\it a priori hypothesis}, that is, to show \eqref{3.2}.
According to the momentum equation $\eqref{a1}_2$, we observe that
\begin{equation*}
  \divv\mathbf{u}=\frac{-(-\Delta)^{-1}\divv(\rho\dot{\mathbf{u}}-\mathbf{B}\cdot\nabla\mathbf{B})+P(\rho)-\bar{P}+\frac12|\mathbf{B}|^2}{2\mu+\lambda},
\end{equation*}
which further inspires us to establish the estimates on the material derivative of the velocity
once the difficulties associated with the magnetic field are overcome
(see Lemmas $\ref{l3.2}$ and $\ref{l3.4}$).

However, since the bulk viscosity is often coupled with the divergence of the velocity, it seems hard to derive the $L^\infty(0,\min\{1,T\};L^2)$-norm for the material derivative of velocity (see Lemma $\ref{l3.4}$) in order to isolate $\lambda$ (see, e.g., \eqref{3.23}--\eqref{3.29}).
In contrast to prior literature concerning slip boundary problems, we have to apply Hoff's strategy to the equation
\begin{equation*}
\rho\dot{\mathbf{u}}-(2\mu+\lambda)\nabla\divv \mathbf{u}+\mu\nabla^\bot\curl \mathbf{u}+\nabla P=\mathbf{B}\cdot\nabla\mathbf{B}-\frac12\nabla|\mathbf{B}|^2
\end{equation*}
rather than
\begin{equation*}
  \rho\dot{\mathbf{u}}-\nabla F+\mu\nabla^\bot\curl\mathbf{u}= \mathbf{B}\cdot\nabla\mathbf{B},
\end{equation*}
thereby giving rise to numerous intractable terms in the absence of the effective viscous flux,
particularly those involving the pressure and the magnetic field. Therefore, we need new approaches to address these issues, as detailed below.

First, it is hard to obtain the estimates for the term
\begin{equation*}
  (2\mu+\lambda)\sigma\int\dot{u}^j\left(\partial_j\divv\mathbf{u}_t +\divv(\mathbf{u}\partial_j\divv\mathbf{u})\right)\mathrm{d}\mathbf{x}.
\end{equation*}
Instead, motivated by \cite{HWZ}, we consider (see \eqref{3.23})
\begin{equation*}
  (2\mu+\lambda)\int(\divf\mathbf{u})^2\mathrm{d}\mathbf{x}~~~~
\text{rather than}
~~~~(2\mu+\lambda)\int(\divv\dot{\mathbf{u}})^2\mathrm{d}\mathbf{x}.
\end{equation*}
But this will make the estimates more delicate and complicated (see \eqref{3.24}--\eqref{3.28}),
especially concerning the additional boundary integrals (see \eqref{3.31}--\eqref{3.32}).

Second, as mentioned above, we have to use the Hodge-type decomposition to isolate the bulk viscosity, where only partial information (curl-free part) from $\|\nabla\mathbf{u}\|_{L^p}$ can be controlled in combination
with $\lambda$. The uniform-in-$\lambda$ estimate for the divergence-free part is a core challenge.
An important consideration is to avoid the occurrence of boundary integrals like
\begin{equation*}
  \int_{\partial\Omega} F(\mathcal{P}\mathbf{u})_t\cdot\nabla(\mathcal{P}\mathbf{u})\cdot\mathbf{n}\mathrm{d}s \ \
   \text{and} \ \ \int_{\partial\Omega} F(\mathcal{P}\mathbf{u})\cdot\nabla(\mathcal{P}\mathbf{u})_t\cdot\mathbf{n}\mathrm{d}s
\end{equation*}
which represents one of our primary objectives in \eqref{3.28} without integration by parts.
This necessitate obtaining uniform bounds of $\|F\|_{L^4}$ (or even $\|F\|_{H^1}$ for boundary integrals).
More precisely, if we turn our attention to the viscosity occurred in the effective viscous flux $F$, the Poincar\'{e} inequality effectively suppresses the power of $\lambda$ in the
isentropic compressible Navier--Stokes equations (where $\bar{F}=0$ with $\mathbf{B}=0$) as that in our previous work \cite{WWZ2}. Owing to the presence of magnetic field, we adopt an indirect approach to the
uniform-in-$\lambda$ estimates (see \eqref{2.10}):
\begin{align*}
 \|F\|_{H^1}\leq C\big(\|\nabla F\|_{L^2}+\|\nabla G\|_{L^2}+\|\mathbf{B}\|_{L^4}^2\big)\leq C\big(\|\nabla F\|_{L^2}+\|\nabla |\mathbf{B}|^2\|_{L^2}+\|\mathbf{B}\|_{L^4}^2\big),
\end{align*}
where $G=F+\frac12|\mathbf{B}|^2
=(2\mu+\lambda)\divv\mathbf{u}-(P(\rho)-\bar{P})$ satisfies $\bar{G}=0$.
Meanwhile, the higher regularity of weak solutions obtained here allows us to work with $\|\nabla^2\mathcal{P}\mathbf{u}\|_{L^2}$ instead of the more problematic $\|\nabla^2\mathbf{u}\|_{L^2}$, resulting in no boundary terms with $\mathcal{P}\mathbf{u}$ in \eqref{3.28}.

Last but not least, the additional presence of the boundary integral terms is analytically intractable because density is not defined on the boundary, especially involving the magnetic field, and the trace theorem will fail in some cases (see \eqref{3.32}). So we have to address the divergence theorem and achieve cancelation among a subset of boundary terms (see \eqref{3.31} and \eqref{3.32}),
circumventing the estimates for some terms such as
\begin{equation*}
  \int_{\partial\Omega} P\dot{\mathbf{u}}\cdot\nabla\mathbf{u}\cdot\mathbf{n}\mathrm{d}s,
  ~~\int_{\partial\Omega}P_t\dot{\mathbf{u}}\cdot\mathbf{n}\mathrm{d}s,~~
  \text{and}~~\int\mathbf{B}\cdot\mathbf{B}_t\dot{\mathbf{u}}\cdot\mathbf{n}\mathrm{d}s.
\end{equation*}
Moreover, an observation of
\begin{equation*}
-\int_{\partial\Omega}\bar{P}_t\dot{\mathbf{u}}\cdot\mathbf{n}\mathrm{d}s=
(\gamma-1)\overline{P\divv\mathbf{u}}\int_{\partial\Omega}\dot{\mathbf{u}}\cdot\mathbf{n}\mathrm{d}s
\end{equation*}
leads us to consider (see \eqref{3.31})
\begin{equation*}
\int_{\partial\Omega}(F-\bar{P})_t\dot{\mathbf{u}}\cdot\mathbf{n}\mathrm{d}s~~~~
\text{rather than}
~~~~\int_{\partial\Omega}F_t\dot{\mathbf{u}}\cdot\mathbf{n}\mathrm{d}s.
\end{equation*}
Furthermore, the divergence theorem and integration by parts enable us to avoid estimating some second-order spatial derivatives for the effective viscous flux (see \eqref{3.32}).

With these obstacles resolved,
we then succeed in deriving the desired estimates on $L^\infty(0,\min\{1,T\};L^2)$-norm (see Lemma $\ref{l3.4}$).
Having these time-independent estimates at hand, we can complete the proof of \textit{a priori hypothesis} by applying Lagrangian
coordinates technique used in \cite{DE97} (see Lemma $\ref{l3.5}$). It should be emphasized that the effective viscous flux and Desjardins-type logarithmic interpolation inequality play essential roles in our analysis.

The rest of the paper is organized as follows. In the next section we recall some known facts and elementary inequalities that will be used later. Section \ref{sec3} is devoted to obtaining {\it a priori} estimates. Then we give the proof of Theorem \ref{t1.1} in Section \ref{sec4}, while the proof of Theorem \ref{t1.2} is carried over to the last section.

\section{Preliminaries}\label{sec2}

This section compiles some well-known facts and elementary inequalities.

\subsection{Auxiliary results and inequalities}
In this subsection we review some known facts and inequalities.
We begin with a lemma concerning the local existence and the possible breakdown of strong solutions to the problem \eqref{a1}--\eqref{a3}, which has been proven in \cite{DW15}.
\begin{lemma}\label{l2.1}
Assume that
\begin{equation*}
\rho_0\in H^{2},~~(\mathbf{u}_0,\mathbf{B}_0)\in H_\omega^{2},~~\divv \mathbf{B}_0=0, ~~\text{and}~~\inf\limits_{\mathbf{x}\in\Omega}\rho_0(\mathbf{x})>0,
\end{equation*}
then there exists a positive constant $T$ such that the initial-boundary value problem \eqref{a1}--\eqref{a3} admits a unique strong solution $(\rho,\mathbf{u},\mathbf{B})$ satisfying
\begin{equation*}
(\rho,\mathbf{u},\mathbf{B})\in C([0,T]; H^{2}),~~\text{and}~~ \inf_{\Omega\times[0,T]}\rho(\mathbf{x},t)\geq\frac{1}{2}
\inf_{\mathbf{x}\in\Omega}\rho_0(\mathbf{x})>0.
\end{equation*}
Moreover, if $T^*$ is the maximal time of existence, then it holds that
\begin{equation*}
\lim\sup_{T\nearrow T^*}\|\rho\|_{L^\infty(0,T;L^\infty)}=\infty.
\end{equation*}
\end{lemma}

The following Gagliardo--Nirenberg inequality (see \cite[Remark 2.1]{LWZ}) will be used frequently later.
\begin{lemma}\label{GN}
(Gagliardo--Nirenberg inequality, special case).
Assume that $\Omega$ is a bounded Lipschitz domain in $\mathbb{R}^2$.
For $p\in [2, \infty)$, $q\in(1, \infty)$, and $r\in (2, \infty)$, there exist generic constants $C_i>0\ (i\in\{1,2,3,4\})$ which may depend only on $p$, $q$, $r$, and $\Omega$ such that, for $f\in H^1$ and $g\in L^q$ with $\nabla g\in L^{r}$,
\begin{gather*}
\|f\|_{L^p}\leq C_1\|f\|_{L^2}^\frac{2}{p}\|\nabla f\|_{L^2}^{1-\frac{2}{p}}+C_2\|f\|_{L^2},\\
\|g\|_{L^\infty}\leq C_3\|g\|_{L^q}^\frac{q(r-2)}{2r+q(r-2)}\|\nabla g\|_{L^r}^\frac{2r}{2r+q(r-2)}+C_4\|g\|_{L^2}.
\end{gather*}
Moreover, if $\int_{\Omega}f(\mathbf{x})\mathrm{d}\mathbf{x}=0$ or $(f\cdot n)|_{\partial\Omega}=0$, we can choose $C_2=0$. Similarly, the constant $C_4=0$ provided $\int_{\Omega}g(\mathbf{x})\mathrm{d}\mathbf{x}=0$ or $(g\cdot n)|_{\partial\Omega}=0$.
\end{lemma}

Next, we present a generalized Poincar{\'e}'s inequality (cf. \cite[Lemma 8]{BS2012}).
\begin{lemma}\label{PO}
Let $\Omega\subset\mathbb{R}^2$ be a bounded Lipschitz domain. Then, for $1<p<\infty$, there exists a positive constant $C$ depending only on $p$ and $\Omega$ such that
\begin{equation*}
\|f\|_{L^p} \leq  C\|\nabla f\|_{L^p},
\end{equation*}
for each vector field $f\in W^{1,p}(\Omega)$ satisfying either $\int_{\Omega}f(\mathbf{x})\mathrm{d}\mathbf{x}=0$ or $(f\cdot n)|_{\partial\Omega}=0$.
\end{lemma}

Moreover, the following Desjardins-type logarithmic interpolation inequality in 2D bounded convex domains has been proven in \cite[Lemma 2.4]{WWZ2}, which extends the torus $\mathbb{T}^2$ case treated in \cite[Lemma 2]{DE97} (see also \cite[Lemma 1]{D1997}).
\begin{lemma}\label{log}
Let $\Omega\subset\mathbb{R}^2$ be a bounded convex domain with smooth boundary.
Assume that $0\le\rho\le\hat{\rho}$ and $\mathbf{u}\in H^1(\Omega)$, then
\begin{equation*}
\|\sqrt\rho\mathbf{u}\|_{L^4}^2\leq C(\hat{\rho},\Omega)(1+\|\sqrt\rho\mathbf{u}\|_{L^2})\|\nabla\mathbf{u}\|_{L^2}
\sqrt{\ln\left(2+\|\nabla\mathbf{u}\|_{L^2}^2\right)},
\end{equation*}
where and in what follows we sometimes write $C(f)$ to emphasize the dependence on $f$.
\end{lemma}

In addition, the following Hodge-type decomposition is given in \cite{WW92,PL96}.
\begin{lemma}\label{Hodge}
Let $1<q<\infty$ and $\Omega$ be a simply connected bounded domain in $\mathbb{R}^2$ with Lipschitz boundary $\partial\Omega$ (e.g., a bounded convex domain). For $\mathbf{u}\in W^{1,q}$ satisfying $(\mathbf{u}\cdot\mathbf{n})|_{\partial\Omega}=0$, there exists a constant $C=C(q, \Omega)>0$ such that
\begin{equation*}
    \|\nabla \mathbf{u}\|_{L^q}\leq C(\|\divv \mathbf{u}\|_{L^q}+\|\curl \mathbf{u}\|_{L^q}).
\end{equation*}
\end{lemma}

Finally, we recall the following commutator estimates in \cite[Lemma 4.3]{F04}, which play an important role in the mollifier arguments.
\begin{lemma}\label{lcom}
Let $\Omega\subset \mathbb{R}^2$ be a domain.
Let $\rho \in L^{p}(\Omega)$ and ${\bf u} \in W^{1,q}(\Omega)$ be given functions such that $1\le p,q<\infty$ and $\frac{1}{p}+\frac{1}{q}\le 1$. For any $\epsilon>0$, then we have
\begin{equation*}
\|\divv[\rho{\bf u}]_\epsilon-\divv\left([\rho]_\epsilon{\bf u}\right)\|_{L^1(K)}\le C(K)\|\rho\|_{L^p(\Omega)}\|{\bf u}\|_{W^{1,q}(\Omega)},
\end{equation*}
and
\begin{equation*}
\divv[\rho{\bf u}]_\epsilon-\divv\left([\rho]_\epsilon{\bf u}\right)\rightarrow 0\ \text{ in}\ \ L^1(K)\ \ \text{as}\ \epsilon\rightarrow0
\end{equation*}
for any compact set $K\subset \Omega$.
\end{lemma}

\subsection{Uniform estimates for $F$, $\boldsymbol\omega$, $\nabla\mathbf{u}$, and $\dot{\mathbf{u}}$}

The momentum equations \eqref{a1}$_2$ can be rewritten as
\begin{equation}\label{2.1}
\rho{\dot{\mathbf{u}}}-\mathbf{B}\cdot\nabla\mathbf{B}=\nabla F-\mu\nabla^\bot\omega,
\end{equation}
which combined with the boundary conditions \eqref{a3} yields that
\begin{align}\label{2.2}
\begin{cases}
-\mu\Delta\omega=\nabla^\bot\cdot(\rho\dot{\mathbf{u}}-\mathbf{B}\cdot\nabla\mathbf{B}), &\mathbf{x}\in\Omega,\\
\omega=0, &\mathbf{x}\in\partial\Omega.
\end{cases}
\end{align}
Thus, we obtain the following estimates by employing Gagliardo--Nirenberg inequality, the standard $L^p$-estimate for elliptic system \eqref{2.2}, the Hodge-type decomposition, and the fact that $2\mu+\lambda\geq\mu>0$. Their proofs are straightforward adaptations of arguments in \cite{WWZ2}, except for the analysis of $\|F\|_{L^p}$.
\begin{lemma}\label{E0}
Let $(\rho,\mathbf{u},\mathbf{B})$ be a smooth solution to the problem \eqref{a1}--\eqref{a3}. Then, for any $2\leq p<\infty$, there exists a generic positive constant $C$ depending only on $p$, $\mu$, and $\Omega$ such that
\begin{equation}\label{E1}
\|\nabla \mathbf{B}\|_{L^p}\leq C\|\curl\mathbf{B}\|_{L^p},
\end{equation}
\begin{equation}\label{E2}
\|\nabla F\|_{L^p}+\|\nabla^2\mathcal{P}\mathbf{u}\|_{L^p}+\|\nabla\omega\|_{L^p}\leq
C\big(\|\rho\dot{\mathbf{u}}\|_{L^p}+\|\nabla \mathbf{u}\|_{L^p}+\|\mathbf{B}\cdot\nabla\mathbf{B}\|_{L^p}\big),
\end{equation}
\begin{equation}\label{E3}
\|\nabla\mathcal{P}\mathbf{u}\|_{L^p}+
\|\omega\|_{L^p}\leq C\big(\|\rho\dot{\mathbf{u}}\|_{L^2}+\|\mathbf{B}\cdot\nabla\mathbf{B}\|_{L^2}\big)^{1-\frac{2}{p}}\|\nabla \mathbf{u}\|_{L^2}^\frac{2}{p}+C\|\nabla \mathbf{u}\|_{L^2},
\end{equation}
\begin{equation}\label{E4}
\|F\|_{H^1}\leq
C\big(\|\rho\dot{\mathbf{u}}\|_{L^2}+\|\nabla \mathbf{u}\|_{L^2}+\||\mathbf{B}||\nabla\mathbf{B}|\|_{L^2}+\|\mathbf{B}\|_{L^4}^2\big),
\end{equation}
\begin{equation}\label{E5}
\|F\|_{L^p}\leq
 C\big[(2\mu+\lambda)\|\divv\mathbf{u}\|_{L^2}+\|P-\bar{P}\|_{L^2}\big]^\frac{2}{p}
 \big(\|\rho\dot{\mathbf{u}}\|_{L^2}+\|\nabla \mathbf{u}\|_{L^2}+\||\mathbf{B}||\nabla\mathbf{B}|\|_{L^2}\big)^{1-\frac{2}{p}}
 +C\|\mathbf{B}\|_{L^{2p}}^2,
\end{equation}
\begin{equation}\label{E6}
\begin{aligned}[b]
\|\nabla \mathbf{u}\|_{L^p}
&\leq
C\big(\|\rho\dot{\mathbf{u}}\|_{L^2}+\||\mathbf{B}||\nabla\mathbf{B}|\|_{L^2}\big)
^{1-\frac{2}{p}}\|\nabla \mathbf{u}\|_{L^2}^\frac{2}{p}+C\|\nabla \mathbf{u}\|_{L^2}
+\frac{C}{2\mu+\lambda}\big(\|P-\bar{P}\|_{L^p}+\|\mathbf{B}\|_{L^{2p}}^2\big)\\
&\quad+\frac{C}{2\mu+\lambda}\|P-\bar{P}\|_{L^2}^\frac{2}{p}
\big(\|\rho\dot{\mathbf{u}}\|_{L^2}
+\|\nabla \mathbf{u}\|_{L^2}
+\||\mathbf{B}||\nabla\mathbf{B}|\|_{L^2}\big)^{1-\frac{2}{p}}.
\end{aligned}
\end{equation}
\end{lemma}
\begin{proof}
As stated above, we only focus on the analysis of $\|F\|_{L^p}$ in what follows. To this end, let
\begin{equation}\label{2.9}
  G\triangleq F+\frac12|\mathbf{B}|^2=(2\mu+\lambda)\divv\mathbf{u}-(P(\rho)-\bar{P}),
\end{equation}
then one sees that $\bar{G}=0$. This along with Sobolev's, Poincar{\'e}'s, and Gagliardo--Nirenberg inequalities implies that
\begin{align}\label{2.10}
 \|F\|_{L^p}&\leq C\|F\|_{H^1} \notag \\
 &\leq C\big(\|\nabla F\|_{L^2}+\|\nabla G\|_{L^2}+\|\mathbf{B}\|_{L^4}^2\big)\notag\\
 &\leq C\big(\|\nabla F\|_{L^2}+\|\nabla |\mathbf{B}|^2\|_{L^2}+\|\mathbf{B}\|_{L^4}^2\big)\notag\\
 &\leq C\big(\|\rho\dot{\mathbf{u}}\|_{L^2}+\|\nabla \mathbf{u}\|_{L^2}+\||\mathbf{B}||\nabla\mathbf{B}|\|_{L^2}+\|\mathbf{B}\|_{L^4}^2\big),
\end{align}
 and
\begin{align}\label{2.11}
 \|F\|_{L^p}&\leq C\big(\|G\|_{L^p}+\|\mathbf{B}\|_{L^{2p}}^2\big) \notag \\
 &\leq C\Big(\|G\|_{L^2}^{\frac{2}{p}}\|\nabla G\|_{L^2}^{1-\frac{2}{p}}+\|\mathbf{B}\|_{L^{2p}}^2\Big)\notag\\
 &\leq C\big[(2\mu+\lambda)\|\divv\mathbf{u}\|_{L^2}+\|P-\bar{P}\|_{L^2}\big]^\frac{2}{p}
 \big(\|\nabla F\|_{L^2}+\|\nabla|\mathbf{B}|^2\|_{L^2}\big)^{1-\frac{2}{p}}
 +C\|\mathbf{B}\|_{L^{2p}}^2\notag\\&\leq
 C\big[(2\mu+\lambda)\|\divv\mathbf{u}\|_{L^2}+\|P-\bar{P}\|_{L^2}\big]^\frac{2}{p}
 \big(\|\rho\dot{\mathbf{u}}\|_{L^2}+\|\nabla \mathbf{u}\|_{L^2}+\||\mathbf{B}||\nabla\mathbf{B}|\|_{L^2}\big)^{1-\frac{2}{p}} \notag \\ & \quad +C\|\mathbf{B}\|_{L^{2p}}^2,
\end{align}
as the desired \eqref{E4} and \eqref{E5}, which combined with \eqref{E3} and Lemma \ref{Hodge} indicates \eqref{E6} owing to
\begin{equation*}
\|\nabla\mathbf{u}\|_{L^p}\leq \frac{C}{2\mu+\lambda}
\big(\|F\|_{L^p}+\|P-\bar{P}\|_{L^p}+\|\mathbf{B}\|_{L^{2p}}^2\big)
+C\|\omega\|_{L^p}.
\tag*{\qedhere}
\end{equation*}
\end{proof}

Next, we retell the boundary condition \eqref{a3}. Indeed,
$(\mathbf{u}\cdot \mathbf{n})|_{\partial\Omega}=0$ yields that
\begin{equation*}
\mathbf{u}=(\mathbf{u}\cdot\mathbf{n}^\bot)\mathbf{n}^\bot, ~~\text{on}~\partial\Omega,
\end{equation*}
where $\mathbf{n}^\bot\triangleq(n^2,-n^1)$ is the unit tangent vector along the boundary.
Then it follows that
\begin{align}\label{2.12}
  \dot{\mathbf{u}}\cdot\mathbf{n}&=\mathbf{u}\cdot \nabla\mathbf{u}\cdot\mathbf{n}
  =\mathbf{u}\cdot\nabla(\mathbf{u}\cdot\mathbf{n})-\mathbf{u}\cdot \nabla\mathbf{n}\cdot\mathbf{u}
  =-\mathbf{u}\cdot \nabla\mathbf{n}\cdot\mathbf{u}\\
  &=-(\mathbf{u}\cdot\mathbf{n}^\bot)\mathbf{u}\cdot \nabla\mathbf{n}\cdot\mathbf{n}^\bot
  =(\mathbf{u}\cdot\mathbf{n}^\bot)\mathbf{u}\cdot \nabla\mathbf{n}^\bot\cdot\mathbf{n}
  ,~~\text{on}~\partial\Omega,\notag
\end{align}
and thus,
\begin{equation*}
[\dot{\mathbf{u}}-(\mathbf{u}\cdot\mathbf{n}^\bot)(\mathbf{u}\cdot \nabla)\mathbf{n}^\bot]\cdot\mathbf{n}=0,~~ \text{on}~ \partial\Omega,
\end{equation*}
which together with Poincar\'{e}'s inequality ensures the following estimates analogous to \cite[Lemma 2.8]{LWZ}.
\begin{lemma}\label{ldot}
Let $(\rho,\mathbf{u},\mathbf{B})$ be a smooth solution to the problem \eqref{a1}--\eqref{a3}. For any $p\in[1,\infty)$, there is a positive constant $C$ which may depend only on $p$ and $\Omega$ such that
\begin{gather*}
\|\dot{\mathbf{u}}\|_{L^p}\leq C\left(\|\nabla\dot{\mathbf{u}}\|_{L^2}+\|\nabla \mathbf{u}\|_{L^2}^2\right),\\
\|\nabla\dot{\mathbf{u}}\|_{L^2}\leq C\big(\|\divv\dot{\mathbf{u}}\|_{L^2}+\|\curl \dot{\mathbf{u}}\|_{L^2}+\|\nabla \mathbf{u}\|_{L^4}^2\big).
\end{gather*}
\end{lemma}

\section{\textit{A priori} estimates}\label{sec3}

This section provides some necessary {\it a priori} bounds for strong solutions guaranteed by Lemma $\ref{l2.1}$. It should be noted that these bounds are independent of the bulk viscosity $\lambda$, the lower bound of $\rho$, the initial regularity, and the time of existence. Specifically, let $T>0$ be fixed and $(\rho, \mathbf{u}, \mathbf{B})$ be the strong solution to \eqref{a1}--\eqref{a3} on $\Omega\times(0, T]$, the following key {\it a priori} estimates on $(\rho, \mathbf{u}, \mathbf{B})$ will be established.
\begin{proposition}\label{p3.1}
Under the conditions of Theorem $\ref{t1.1}$, if $(\rho, \mathbf{u}, \mathbf{B})$ is a strong solution to initial-boundary value problem \eqref{a1}--\eqref{a3} satisfying
\begin{align}\label{3.1}
\sup_{\Omega\times[0,T]}\rho\le2\hat{\rho},\ \ \frac{1}{(2\mu+\lambda)^2}\int_{0}^{T}
\|P-\bar{P}\|_{L^4}^4\mathrm{d}t\le2,
\end{align}
then one has that
\begin{align}\label{3.2}
\sup_{\Omega\times[0,T]}\rho\le\frac{7}{4}\hat{\rho},\ \ \frac{1}{(2\mu+\lambda)^2}\int_{0}^{T}
\|P-\bar{P}\|_{L^4}^4\mathrm{d}t\le1.
\end{align}
\end{proposition}

Before proving Proposition \ref{p3.1}, we show some necessary \textit{a priori} estimates, see Lemmas \ref{l3.1}--\ref{l3.4} below. Let us start with the elementary energy estimate of $(\rho, \mathbf{u}, \mathbf{B})$.
\begin{lemma}\label{l3.1}
It holds that
\begin{align*}
\sup_{0\le t\leq T}\int\bigg(\frac{1}{2}\rho |\mathbf{u}|^2+\frac{1}{2} |\mathbf{B}|^2+G(\rho)\bigg)\mathrm{d}\mathbf{x}+\int_0^T\big[(2\mu+\lambda)\|\divv \mathbf{u}\|_{L^2}^2+\mu\|\curl \mathbf{u}\|_{L^2}^2+\nu\|\curl \mathbf{B}\|_{L^2}^2\big]\mathrm{d}t\leq C_0.
\end{align*}
\end{lemma}
\begin{proof}
Using $\eqref{a1}_1$ and the definition of $G(\rho)$ in \eqref{1.4}, we have that
\begin{equation}\label{3.3}
(G(\rho))_t+\divv (G(\rho)\mathbf{u})+(P-P(\bar{\rho}))\divv \mathbf{u}=0.
\end{equation}
Multiplying $\eqref{a1}_2$ by $\mathbf{u}$ and $\eqref{a1}_3$ by $\mathbf{B}$, respectively, adding the summation to \eqref{3.3}, and integrating (by parts) the resultant over $\Omega$, we deduce from
the vector identity
\begin{equation}\label{3.4}
\Delta \mathbf{u}=\nabla\divv\mathbf{u}-\nabla^\bot\curl\mathbf{u}
\end{equation}
and the boundary conditions \eqref{a3} that
\begin{align}\label{3.5}
\frac{\mathrm{d}}{\mathrm{d}t}\int\bigg(\frac{1}{2}\rho|\mathbf{u}|^2+\frac{1}{2} |\mathbf{B}|^2+G(\rho)\bigg)\mathrm{d}\mathbf{x}+\int\big[(2\mu+\lambda)(\divv\mathbf{u})^2
+\mu(\curl\mathbf{u})^2+\nu(\curl\mathbf{B})^2\big]\mathrm{d}\mathbf{x}
=0.
\end{align}
Integrating \eqref{3.5} with respect to $t$ over $(0,T)$ completes the proof.
\end{proof}

The next lemma provides a time-independent $L^\infty(0, T;L^2)$ estimate for $\divv \mathbf{u}$, $\curl \mathbf{u}$, and $\curl \mathbf{B}$.
\begin{lemma}\label{l3.2}
Let \eqref{3.1} be satisfied, then there exists a positive number $D_2$ depending only on $\Omega$, $\hat{\rho}$, $a$, $\gamma$, $\nu$, and $\mu$ such that
\begin{align*}
 &\sup_{0\leq t\leq T}\int\big[(2\mu+\lambda)(\divv \mathbf{u})^2+\mu(\curl \mathbf{u})^2+\nu(\curl \mathbf{B})^2\big]\mathrm{d}\mathbf{x}\notag\\
 &\quad +\int_0^T\big(\|\sqrt{\rho}{\dot{\mathbf{u}}}\|_{L^2}^2+\|\nabla\curl \mathbf{B}\|_{L^2}^2+\|\mathbf{B}_t\|_{L^2}^2\big)
\mathrm{d}t\leq (2+M)^{\exp\big\{2D_2(1+C_0)^3\big\}}
\end{align*}
provided that $\lambda$ satisfies \eqref{lam} with $D\geq D_2$.
\end{lemma}
\begin{proof}
According to $\eqref{a1}_1$ and \eqref{3.4}, we can rewrite $\eqref{a1}_2$ as
\begin{equation}\label{3.6}
\rho\mathbf{u}_t+\rho\mathbf{u}\cdot\nabla\mathbf{u}
-(2\mu+\lambda)\nabla\divv \mathbf{u}+\mu\nabla^\bot\curl \mathbf{u}
+\nabla P=\mathbf{B}\cdot\nabla\mathbf{B}-\frac12\nabla|\mathbf{B}|^2.
\end{equation}
Multiplying \eqref{3.6} by $\mathbf{u}_t$ and integration by parts, one gets from \eqref{a3} that
\begin{align}\label{3.7}
&\frac{1}{2}\frac{\mathrm{d}}{\mathrm{d}t}\int\big[(2\mu+\lambda)(\divv \mathbf{u})^2+\mu(\curl \mathbf{u})^2-|\mathbf{B}|^2\divv
\mathbf{u}+2\mathbf{B}\cdot\nabla\mathbf{u}\cdot\mathbf{B}\big]
\mathrm{d}\mathbf{x}+\int\rho |\dot{\mathbf{u}}|^2\mathrm{d}\mathbf{x}\notag\\
&=-\int\mathbf{u}_t\cdot\nabla P\mathrm{d}\mathbf{x}+\int\rho \dot{\mathbf{u}}\cdot(\mathbf{u}\cdot\nabla)\mathbf{u}\mathrm{d}\mathbf{x}
+\int\big(\mathbf{B}_t\cdot\nabla\mathbf{u}\cdot\mathbf{B}+\mathbf{B}\cdot\nabla\mathbf{u}\cdot\mathbf{B}_t
-\mathbf{B}\cdot\mathbf{B}_t\divv\mathbf{u}\big)\mathrm{d}\mathbf{x}.
\end{align}
From $\eqref{a1}_3$ and \eqref{a3}, we have that
\begin{align*}
\int|\mathbf{B}\cdot\nabla\mathbf{u}-\mathbf{u}\cdot\nabla\mathbf{B}
-\mathbf{B}\divv\mathbf{u}|^2\mathrm{d}\mathbf{x}
&=\int|\mathbf{B}_t-\nu\Delta\mathbf{B}|^2\mathrm{d}\mathbf{x}\notag\\
&=\frac{\mathrm{d}}{\mathrm{d}t}\int\nu(\curl \mathbf{B})^2\mathrm{d}\mathbf{x}+
\int\big(\nu^2|\nabla\curl \mathbf{B}|^2+|\mathbf{B}_t|^2\big)\mathrm{d}\mathbf{x},
\end{align*}
which together with \eqref{3.7} yields that
\begin{align}\label{3.8}
 &\frac{1}{2}\frac{\mathrm{d}}{\mathrm{d}t}\int\big[(2\mu+\lambda)(\divv \mathbf{u})^2+\mu(\curl \mathbf{u})^2+2\nu(\curl \mathbf{B})^2-|\mathbf{B}|^2\divv
\mathbf{u}+2\mathbf{B}\cdot\nabla\mathbf{u}\cdot\mathbf{B}\big]\mathrm{d}\mathbf{x}\notag\\
&\quad+\int(\rho |\dot{\mathbf{u}}|^2+\nu^2|\nabla\curl \mathbf{B}|^2+|\mathbf{B}_t|^2)\mathrm{d}\mathbf{x}\notag\\
 &=-\int\mathbf{u}_t\cdot\nabla P\mathrm{d}\mathbf{x}+\int\rho \dot{\mathbf{u}}\cdot(\mathbf{u}\cdot\nabla)\mathbf{u}\mathrm{d}\mathbf{x}
 +\int(\mathbf{B}_t\cdot\nabla\mathbf{u}\cdot\mathbf{B}+\mathbf{B}\cdot\nabla\mathbf{u}\cdot\mathbf{B}_t
-\mathbf{B}\cdot\mathbf{B}_t\divv\mathbf{u})\mathrm{d}\mathbf{x}
 \notag\\
&\quad+\int|\mathbf{B}\cdot\nabla\mathbf{u}-\mathbf{u}\cdot\nabla\mathbf{B}-\mathbf{B}\divv\mathbf{u}|^2\mathrm{d}\mathbf{x}
\triangleq\sum_{i=1}^{4}\mathcal{I}_i.
\end{align}

We then proceed to bound each $\mathcal{I}_i$. It can be deduced from $\eqref{a1}_1$ and \eqref{1.7} that
\begin{align}\label{3.9}
&\mathcal{I}_1=\int P\divv\mathbf{u}_t\mathrm{d}\mathbf{x}=\frac{\mathrm{d}}{\mathrm{d}t}\int P\divv\mathbf{u} \mathrm{d}\mathbf{x}-\int P'(\rho)\rho_t\divv\mathbf{u}\mathrm{d}\mathbf{x}\notag\\
&=\frac{\mathrm{d}}{\mathrm{d}t}\int P\divv\mathbf{u}\mathrm{d}\mathbf{x}+
\int P'(\rho)\rho(\divv\mathbf{u})^2\mathrm{d}\mathbf{x}+
\int\mathbf{u}\cdot\nabla(P-\bar{P})\divv\mathbf{u}\mathrm{d}\mathbf{x}\notag\\
&=\frac{\mathrm{d}}{\mathrm{d}t}\int P\divv\mathbf{u}\mathrm{d}\mathbf{x}+
\int(P'(\rho)\rho-P+\bar{P})(\divv\mathbf{u})^2\mathrm{d}\mathbf{x}-
\int(P-\bar{P})\mathbf{u}\cdot\nabla\divv\mathbf{u}\mathrm{d}\mathbf{x}\notag\\
&=\frac{\mathrm{d}}{\mathrm{d}t}\int P\divv\mathbf{u}\mathrm{d}\mathbf{x}+
\int(P'(\rho)\rho-P+\bar{P})(\divv\mathbf{u})^2\mathrm{d}\mathbf{x}-\frac{1}{2\mu+\lambda}
\int (P-\bar{P})\mathbf{u}\cdot\nabla\Big(F+P-\bar{P}+\frac12|\mathbf{B}|^2\Big)\mathrm{d}\mathbf{x}\notag\\
&=\frac{\mathrm{d}}{\mathrm{d}t}\int P\divv\mathbf{u}\mathrm{d}\mathbf{x}+
\int(P'(\rho)\rho-P+\bar{P})(\divv\mathbf{u})^2\mathrm{d}\mathbf{x}
+\frac{1}{4\mu+2\lambda}\int(P-\bar{P})^2\divv\mathbf{u}\mathrm{d}\mathbf{x}\notag\\
&\quad-\frac{1}{2\mu+\lambda}
\int (P-\bar{P})\mathbf{u}\cdot\nabla F\mathrm{d}\mathbf{x}
-\frac{1}{2\mu+\lambda}
\int (P-\bar{P})\mathbf{u}\cdot\nabla \mathbf{B}\cdot\mathbf{B}\mathrm{d}\mathbf{x}
\notag\\
&\leq\frac{\mathrm{d}}{\mathrm{d}t}\int P\divv\mathbf{u}\mathrm{d}\mathbf{x}
+\frac{1}{8}\|\sqrt{\rho}\dot{\mathbf{u}}\|_{L^2}^2
+\frac{\nu^2}{8}\|\nabla\curl\mathbf{B}\|_{L^2}^2
+C\|\nabla\mathbf{B}\|_{L^2}^2\big(1+\|\nabla\mathbf{B}\|_{L^2}^2\big)\notag\\&\quad
+C\|\nabla\mathbf{u}\|_{L^{2}}^{2}\big(1+\|\nabla\mathbf{u}\|_{L^{2}}^{2}\big)
+\frac{C}{(2\mu+\lambda)^2}\|P-\bar{P}\|_{L^4}^4,
\end{align}
where we have used
\begin{align*}
&\frac{1}{2\mu+\lambda}\left|
\int(P-\bar{P})\mathbf{u}\cdot\nabla F\mathrm{d}\mathbf{x}\right|
+\frac{1}{2\mu+\lambda}\left|
\int(P-\bar{P})\mathbf{u}\cdot\nabla \mathbf{B}\cdot\mathbf{B}\mathrm{d}\mathbf{x}\right|\notag\\
&\leq\frac{C}{2\mu+\lambda}\|P-\bar{P}\|_{L^\infty}
(\|\mathbf{u}\|_{L^2}
\|\nabla F\|_{L^2}+\|\mathbf{u}\|_{L^4}\|\nabla \mathbf{B}\|_{L^2}\|\mathbf{B}\|_{L^4})\notag\\
&\leq\frac{C}{2\mu+\lambda}\|\nabla\mathbf{u}\|_{L^2}
\big(\|\sqrt{\rho}\dot{\mathbf{u}}\|_{L^2}+\|\nabla\mathbf{u}\|_{L^2}+\||\mathbf{B}||\nabla\mathbf{B}|
\|_{L^2}+\|\nabla\mathbf{B}\|_{L^2}^2\big)
\notag\\
&\leq\frac{1}{8}\|\sqrt{\rho}\dot{\mathbf{u}}\|_{L^2}^2
+\frac{\nu^2}{8}\|\nabla\curl\mathbf{B}\|_{L^2}^2
+C\|\nabla\mathbf{B}\|_{L^2}^2\big(1+\|\nabla\mathbf{B}\|_{L^2}^2\big)
+C\|\nabla\mathbf{u}\|_{L^{2}}^{2}\big(1+\|\nabla\mathbf{u}\|_{L^{2}}^{2}\big)
\end{align*}
due to Lemma \ref{E0}, Poincar\'{e}'s inequality, Gagliardo--Nirenberg inequality, and
\begin{equation*}
  \||\mathbf{B}||\nabla\mathbf{B}|\|_{L^2}\leq \|\mathbf{B}\|_{L^4}\|\nabla\mathbf{B}\|_{L^4}
  \leq C\Big(\|\nabla\mathbf{B}\|_{L^2}^2
  +\|\nabla\mathbf{B}\|_{L^2}^\frac32\|\nabla\curl\mathbf{B}\|_{L^2}^\frac12\Big).
\end{equation*}

Next, using Cauchy--Schwarz, Poincar\'{e}'s, and H\"older's inequalities, we derive from Lemmas $\ref{log}$, $\ref{E0}$, and $\ref{l3.1}$ that
\begin{align}\label{3.10}
&\mathcal{I}_2=\int\rho \dot{\mathbf{u}}\cdot(\mathbf{u}\cdot\nabla)\mathbf{u}\mathrm{d}\mathbf{x}\notag\\
&\leq C\|\sqrt{\rho}\dot{\mathbf{u}}\|_{L^2}\|\sqrt{\rho}\mathbf{u}\|_{L^{4}}
\|\nabla\mathbf{u}\|_{L^{4}}\notag\\
&\leq C\|\sqrt{\rho}\dot{\mathbf{u}}\|_{L^2}
(1+\|\sqrt\rho\mathbf{u}\|_{L^2})^\frac12\|\nabla\mathbf{u}\|_{L^{2}}^\frac12
\ln^\frac14\big(2+\|\nabla\mathbf{u}\|_{L^{2}}^2\big)
\Big[\Big(\|\sqrt{\rho}\dot{\mathbf{u}}\|_{L^2}^{\frac{1}{2}}
+\||\mathbf{B}||\nabla\mathbf{B}|\|_{L^2}^{\frac{1}{2}}\Big)
\|\nabla\mathbf{u}\|_{L^2}^\frac{1}{2}+\|\nabla\mathbf{u}\|_{L^{2}}\notag\\
&\quad
+\frac{1}{2\mu+\lambda}\big(\|P-\bar{P}\|_{L^4}+\|\mathbf{B}\|_{L^{8}}^2\big)
+\frac{1}{2\mu+\lambda}\|P-\bar{P}\|_{L^2}^{\frac{1}{2}}\Big(
\|\sqrt{\rho}\dot{\mathbf{u}}\|_{L^2}^{\frac{1}{2}}
+\||\mathbf{B}||\nabla\mathbf{B}|\|_{L^2}^{\frac{1}{2}}+\|\nabla\mathbf{u}\|_{L^{2}}^\frac12\Big)
\Big]\notag\\
&\leq \frac{1}{8}\|\sqrt{\rho}\dot{\mathbf{u}}\|_{L^2}^2
+\frac{\nu^2}{8}\|\nabla\curl\mathbf{B}\|_{L^{2}}^2
+C\|\nabla\mathbf{B}\|_{L^2}^2\big(1+\|\nabla\mathbf{B}\|_{L^2}^2\big)
+\frac{C}{(2\mu+\lambda)^4}\|P-\bar{P}\|_{L^4}^4\notag\\
&\quad
+C(1+C_0)^2\|\nabla\mathbf{u}\|_{L^2}^2
\big(1+\|\nabla\mathbf{u}\|_{L^{2}}^2+\|\mathbf{B}\|_{L^{4}}^4\big)
\ln\big(2+\|\nabla\mathbf{u}\|_{L^2}^2\big),
\end{align}
where in the last inequality one has used
\begin{gather*}
\||\mathbf{B}||\nabla\mathbf{B}|\|_{L^2}\leq \|\mathbf{B}\|_{L^4}\|\nabla\mathbf{B}\|_{L^4}
\leq CC_0^\frac14
\Big(\|\nabla\mathbf{B}\|_{L^2}^\frac32
+\|\nabla\mathbf{B}\|_{L^2}\|\nabla\curl\mathbf{B}\|_{L^2}^\frac12\Big),\\
\|\mathbf{B}\|_{L^{8}}^2\leq\|\mathbf{B}\|_{L^{\infty}}\|\mathbf{B}\|_{L^{4}}
\leq C\|\mathbf{B}\|_{L^{4}}^\frac32\|\nabla\mathbf{B}\|_{L^{4}}^\frac12
\leq
C\|\mathbf{B}\|_{L^{4}}\Big(\|\nabla\mathbf{B}\|_{L^2}
+C_0^{\frac18}\|\nabla\mathbf{B}\|_{L^2}^\frac12
\|\nabla\curl\mathbf{B}\|_{L^2}^\frac14\Big).
\end{gather*}

For the terms $\mathcal{I}_3$ and $\mathcal{I}_4$, one infers from Gagliardo--Nirenberg and H\"older's inequalities that
\begin{align}\label{3.11}
 \mathcal{I}_3
 &\leq C\|\mathbf{B}\|_{L^\infty}\|\mathbf{B}_t\|_{L^2}\|\nabla\mathbf{u}\|_{L^2}\notag \\
& \leq CC_0^\frac18\Big(\|\nabla\mathbf{B}\|_{L^2}^\frac34
+\|\nabla\mathbf{B}\|_{L^2}^\frac12\|\nabla\curl\mathbf{B}\|_{L^2}^\frac14\Big)\|\mathbf{B}_t\|_{L^2}\|\nabla\mathbf{u}\|_{L^2}\notag\\
 &\leq \frac14\|\mathbf{B}_t\|_{L^2}^2+\frac{\nu^2}{8}\|\nabla\curl\mathbf{B}\|_{L^{2}}^2
 +C\|\nabla\mathbf{B}\|_{L^2}^2\big(1+\|\nabla\mathbf{B}\|_{L^2}^2\big)
 +CC_0^\frac12\|\nabla\mathbf{u}\|_{L^2}^4,\\ \label{3.12}
 \mathcal{I}_4&\leq C\big(\|\mathbf{B}\|_{L^\infty}^2\|\nabla\mathbf{u}\|_{L^2}^2
 +\|\mathbf{u}\|_{L^8}^2\|\nabla\mathbf{B}\|_{L^4}\|\nabla\mathbf{B}\|_{L^2}\big)\notag\\
 &\leq CC_0^{\frac14}\|\nabla\mathbf{u}\|_{L^2}^2
\Big(\|\nabla\mathbf{B}\|_{L^2}^\frac{3}{2}
+\|\nabla\mathbf{B}\|_{L^2}\|\nabla\curl\mathbf{B}\|_{L^2}^\frac{1}{2}\Big)
+C C_0^{\frac14}\|\nabla\mathbf{u}\|_{L^{2}}^\frac{3}{2}\Big(\|\nabla\mathbf{B}\|_{L^2}^2
+\|\nabla\mathbf{B}\|_{L^2}^\frac32\|\nabla\curl\mathbf{B}\|_{L^2}^\frac12\Big)\notag\\
 &\leq \frac{\nu^2}{8}\|\nabla\curl\mathbf{B}\|_{L^{2}}^2
 +C(1+C_0)\|\nabla\mathbf{u}\|_{L^2}^2\big(1+\|\nabla\mathbf{u}\|_{L^{2}}^2\big)
 +C\|\nabla\mathbf{B}\|_{L^2}^2\big(1+\|\nabla\mathbf{B}\|_{L^2}^2\big).
\end{align}

Hence, substituting \eqref{3.9}--\eqref{3.12} into \eqref{3.8}, one deduces from \eqref{E1} that
\begin{align}\label{3.13}
&\frac{1}{2}\frac{\mathrm{d}}{\mathrm{d}t}\mathcal{E}_1(t)
+\int\big(\rho |\dot{\mathbf{u}}|^2+\nu^2|\nabla\curl\mathbf{B}|^2+|\mathbf{B}_t|^2\big)\mathrm{d}\mathbf{x}\notag\\
&\leq
C(1+C_0)^2\big(\|\nabla\mathbf{u}\|_{L^2}^2+\|\curl\mathbf{B}\|_{L^2}^2\big)
\big(1+\|\nabla\mathbf{u}\|_{L^{2}}^2+\|\mathbf{B}\|_{L^4}^4\big)
\ln\big(2+\|\nabla\mathbf{u}\|_{L^2}^2\big)\notag\\
&\quad +C\|\curl\mathbf{B}\|_{L^2}^2\big(1+\|\curl\mathbf{B}\|_{L^2}^2\big)
+\frac{C}{(2\mu+\lambda)^2}\|P-\bar{P}\|_{L^4}^4,
\end{align}
where
\begin{equation}\label{3.14}
\mathcal{E}_1(t)\triangleq\int\big[(2\mu+\lambda)(\divv \mathbf{u})^2+\mu(\curl \mathbf{u})^2+2\nu(\curl \mathbf{B})^2-2P\divv\mathbf{u}-|\mathbf{B}|^2\divv
\mathbf{u}+2\mathbf{B}\cdot\nabla\mathbf{u}\cdot\mathbf{B}\big]\mathrm{d}\mathbf{x}.
\end{equation}
By Lemma \ref{Hodge}, one has that
\begin{equation*}
\bigg|\int(|\mathbf{B}|^2\divv
\mathbf{u}-2\mathbf{B}\cdot\nabla\mathbf{u}\cdot\mathbf{B})\mathrm{d}\mathbf{x}\bigg|\leq C\|\mathbf{B}\|_{L^4}^2\|\nabla\mathbf{u}\|_{L^2}
  \leq C\|\mathbf{B}\|_{L^4}^4+ \frac{\mu}{8}\big(\|\divv\mathbf{u}\|_{L^2}^2+\|\curl\mathbf{u}\|_{L^2}^2\big),
\end{equation*}
which along with \eqref{3.1} implies that there exists $D_1=D_1(a,\gamma,\Omega,\hat{\rho})>0$ such that
\begin{align}\label{3.15}
\mathcal{E}_1(t)\thicksim(2\mu+\lambda)\|\divv\mathbf{u}\|_{L^2}^2+\mu\|\curl\mathbf{u}\|_{L^2}^2
+2\nu\|\curl\mathbf{B}\|_{L^2}^2+\|\mathbf{B}\|_{L^4}^4
\end{align}
provided $\lambda\ge D_1$.

To address $\|\mathbf{B}\|_{L^4}^4$, multiplying $\eqref{a1}_3$ by $4|\mathbf{B}|^2\mathbf{B}$ and integrating the resultant over $\Omega$, one sees that
\begin{align}\label{3.16}
  \frac{\mathrm{d}}{\mathrm{d}t}\|\mathbf{B}\|_{L^4}^4&\leq
  -4\nu\||\mathbf{B}||\curl\mathbf{B}|\|_{L^2}^2
  +C\int|\curl\mathbf{B}||\mathbf{B}\cdot\nabla^\bot|\mathbf{B}|^2|\mathrm{d}\mathbf{x}
  +C\int|\mathbf{B}|^4|\nabla\mathbf{u}|\mathrm{d}\mathbf{x}\notag\\
  &\leq -2\nu\||\mathbf{B}||\curl\mathbf{B}|\|_{L^2}^2+C\int|\mathbf{B}|^2|\nabla\mathbf{B}|^2\mathrm{d}\mathbf{x}
  +C\int|\mathbf{B}|^4|\nabla\mathbf{u}|\mathrm{d}\mathbf{x}\notag\\
  &\leq -2\nu\||\mathbf{B}||\curl\mathbf{B}|\|_{L^2}^2+C\|\mathbf{B}\|_{L^{4}}^2\|\nabla\mathbf{B}\|_{L^{4}}^2
  +C\|\mathbf{B}\|_{L^{\infty}}^2\|\mathbf{B}\|_{L^{4}}^2\|\nabla\mathbf{u}\|_{L^{2}}\notag\\
&\leq-2\nu\||\mathbf{B}||\curl\mathbf{B}|\|_{L^2}^2
+\frac{\nu^2}{8}\|\nabla\curl\mathbf{B}\|_{L^{2}}^2
+C\big(1+\|\mathbf{B}\|_{L^4}^4\big)\big(\|\nabla\mathbf{u}\|_{L^2}^2
+\|\curl\mathbf{B}\|_{L^2}^2\big).
\end{align}
Setting
\begin{align*}
f_1(t)\triangleq2+\mathcal{E}_1(t),~~g_1(t)\triangleq (1+C_0)^2\big(\|\nabla\mathbf{u}\|_{L^2}^2+\|\curl \mathbf{B}\|_{L^2}^2\big)
+\frac{1}{(2\mu+\lambda)^2}\|P-\bar{P}\|_{L^4}^4,
\end{align*}
we thus deduce from \eqref{3.13}, \eqref{3.15}, \eqref{3.16}, and Lemma \ref{Hodge} that
\begin{equation*}
    f'_{1}(t)\leq Cg_{1}(t)f_{1}(t)\ln f_{1}(t),
\end{equation*}
and hence,
\begin{equation*}
    \big(\ln f_1(t)\big)'\leq Cg_1(t)\ln f_1(t).
\end{equation*}
This along with Gronwall's inequality, \eqref{3.1}, and Lemma \ref{l3.1}
implies that there is a positive constant $D_2=D_2(a,\gamma,\hat{\rho},\Omega,\nu,\mu)\geq D_1$ such that
\begin{equation}\label{3.17}
    \sup_{0\leq t\leq T}\big[(2\mu+\lambda)\|\divv \mathbf{u}\|_{L^2}^2+\mu\|\curl \mathbf{u}\|_{L^2}^2+\nu\|\curl \mathbf{B}\|_{L^2}^2+\|\mathbf{B}\|_{L^4}^4\big]
    \leq
    (2+M)^{\exp\big\{D_2(1+C_0)^3\big\}}
\end{equation}
provided $\lambda\ge D_2$.
Integrating \eqref{3.13} with respect to $t$ over $(0,T)$ together with \eqref{3.17} leads to
\begin{align*}
&\int_0^T\big(\|\sqrt{\rho}{\dot{\mathbf{u}}}\|_{L^2}^2+\|\nabla\curl \mathbf{B}\|_{L^2}^2+
\|\mathbf{B}_t\|_{L^2}^2\big)
\mathrm{d}t\notag\\
&\leq C(1+M)+C(1+C_0)^2C_0
(2+M)^{\exp\big\{D_2(1+C_0)^3\big\}}
\ln\bigg\{(2+M)^{\exp\big\{\frac{3}{2}D_2(1+C_0)^3\big\}}
\bigg\}\notag\\
&\leq (2+M)^{\exp\big\{\frac{7}{4}D_2(1+C_0)^3\big\}}
\end{align*}
provided that $\lambda$ satisfies \eqref{lam} with $D\geq D_2$,
which along with \eqref{3.17} concludes the proof.
\end{proof}

Next, we give the bound of $\frac{1}{2\mu+\lambda}\int_{0}^{T}\|P-\bar{P}\|_{L^4}^4\mathrm{d}t$.
\begin{lemma}\label{l3.3}
Let \eqref{3.1} be satisfied, then it holds that
\begin{align}\label{wwz1}
\frac{1}{2\mu+\lambda}\int_{0}^{T}\|P-\bar{P}\|_{L^4}^4\mathrm{d}t\leq
(2+M)^{\exp\big\{3D_2(1+C_0)^3\big\}}
\end{align}
provided that $\lambda$ satisfies \eqref{lam} with $D\geq 2D_2$.
\end{lemma}
\begin{proof}
It follows from $\eqref{a1}_1$ and $P(\rho)=a\rho^\gamma$ that
\begin{equation*}
  P_t+\divv(P\mathbf{u})+(\gamma-1)P\divv\mathbf{u}=0,
\end{equation*}
which leads to
\begin{equation}\label{3.18}
  (P-\Bar{P})_t+\mathbf{u}\cdot\nabla(P-\Bar{P})+\gamma P\divv\mathbf{u}-(\gamma-1)\overline{P\divv\mathbf{u}}=0,
\end{equation}
where
\begin{equation*}
  \overline{P\divv\mathbf{u}}=\fint a\rho^\gamma\divv\mathbf{u}\mathrm{d}\mathbf{x}\leq C(a,\gamma,\hat{\rho},\Omega)\|\divv\mathbf{u}\|_{L^2}.
\end{equation*}

Multiplying $\eqref{3.18}$ by $3(P-\bar{P})^2$ and integrating the resulting equality over $\Omega$, one gets that
\begin{align*}
&\frac{3\gamma-1}{2\mu+\lambda}\left\|P-\bar{P}\right\|_{L^4}^4\notag\\
&=-\frac{\mathrm{d}}{\mathrm{d}t}\int(P-\bar{P})^3\mathrm{d}\mathbf{x}
-\frac{3\gamma-1}{2(2\mu+\lambda)}\int(P-\bar{P})^3\big(2F+|\mathbf{B}|^2\big)
\mathrm{d}\mathbf{x}
-3\gamma \bar{P}\int(P-\bar{P})^2\divv\mathbf{u}\mathrm{d}\mathbf{x}\notag\\
&\quad +3(\gamma-1)\overline{P\divv\mathbf{u}}\int(P-\bar{P})^2\mathrm{d}\mathbf{x}\notag\\
&\leq-\frac{\mathrm{d}}{\mathrm{d}t}\int(P-\bar{P})^3\mathrm{d}\mathbf{x}+
\frac{3\gamma-1}{4(2\mu+\lambda)}\left\|P-\bar{P}\right\|_{L^4}^4
+\frac{C}{2\mu+\lambda}\big(\|F\|_{L^4}^4+\|\mathbf{B}\|_{L^8}^8\big)
+C(2\mu+\lambda)\left\|\divv\mathbf{u}\right\|_{L^2}^2.
\end{align*}
Integrating the above inequality over $(0,T)$, it follows from Lemmas $\ref{l3.1}$, $\ref{l3.2}$, and \eqref{E5} that
\begin{align*}
&\frac{1}{2\mu+\lambda}\int_{0}^{T}\|P-\bar{P}\|_{L^{4}}^{4}\mathrm{d}t\notag\\
&\leq \frac{C}{2\mu+\lambda}\int_0^T
\big[(2\mu+\lambda)^2\|\divv\mathbf{u}\|_{L^2}^2+\|P-\bar{P}\|_{L^2}^2\big]
\big(\|\sqrt{\rho}\dot{\mathbf{u}}\|_{L^2}^2+\|\nabla \mathbf{u}\|_{L^2}^2+\||\mathbf{B}||\nabla\mathbf{B}|\|_{L^2}^2\big)\mathrm{d}t\notag\\
&\quad+C\sup_{0\leq t\leq T}\left\|P-\bar{P}\right\|_{L^3}^3
+\frac{CC_0^\frac12}{2\mu+\lambda}
\int_0^T\big[\|\mathbf{B}\|_{L^4}^4\|\nabla\mathbf{B}\|_{L^2}^2
\big(1+\|\nabla\curl\mathbf{B}\|_{L^2}^2\big)\big]\mathrm{d}t+C(1+C_0)\notag\\
&\leq(2+M)^{\exp\big\{3D_2(1+C_0)^3\big\}},
\end{align*}
as the desired \eqref{wwz1}.
\end{proof}

Motivated by \cite{Hoff95,Hoff95*,HWZ}, we then establish the time-weighted estimate for $\|\sqrt{\rho}{\dot{\mathbf{u}}}\|_{L^2}^2$ and $\|\nabla\curl \mathbf{B}\|_{L^2}^2$.
\begin{lemma}\label{l3.4}
Let \eqref{3.1} be satisfied, then it holds that
\begin{align}\label{3.19}
&\sup_{0\leq t\leq T}\big[\sigma\big(\|\sqrt{\rho}{\dot{\mathbf{u}}}\|_{L^2}^2+\|\nabla\curl \mathbf{B}\|_{L^2}^2+\|\mathbf{B}_t\|_{L^2}^2\big)\big]+(2\mu+\lambda)\int_0^T\sigma\|\divf \mathbf{u}\|_{L^2}^2\mathrm{d}t\notag\\
&\quad+\int_0^T\big(\mu\sigma\|\curl \dot{\mathbf{u}}\|_{L^2}^2+\nu\sigma\|\curl\mathbf{B}_t\|_{L^2}^2\big)\mathrm{d}t
\leq\exp\bigg\{(2+M)^{\exp\big\{4D_2(1+C_0)^3\big\}}\bigg\}
\end{align}
provided that $\lambda$ satisfies \eqref{lam} with $D\geq 3D_2$,
where $\divf\mathbf{u}\triangleq\divv\mathbf{u}_t+\mathbf{u}\cdot\nabla\divv\mathbf{u}$.
\end{lemma}
\begin{proof}
Operating $\sigma\dot{u}^j[\partial/\partial t+\divv({\mathbf{u}}\cdot)]$ on $\eqref{3.6}^j$, summing all the equalities with respect to $j$, and integrating the resultant over $\Omega$, we obtain that
\begin{align}\label{3.20}
&\frac{1}{2}\frac{\mathrm{d}}{\mathrm{d}t}\int\sigma\rho|\dot{\mathbf{u}}|^2\mathrm{d}\mathbf{x}
-\frac{\sigma'}{2}\int\rho|\dot{\mathbf{u}}|^2\mathrm{d}\mathbf{x}\notag\\
&=-\sigma\int\dot{u}^j\big[\partial_j P_{t}+\divv(\mathbf{u}\partial_{j}P)\big]\mathrm{d}\mathbf{x}
-\mu\sigma\int\dot{u}^j\big[(\nabla^\bot\curl \mathbf{u}_{t})^j+\divv\big(\mathbf{u}(\nabla^\bot\curl \mathbf{u})^{j}\big)\big]\mathrm{d}\mathbf{x}\notag\\
&\quad+(2\mu+\lambda)\sigma\int\dot{u}^j\big[\partial_j\divv\mathbf{u}_t  +\divv(\mathbf{u}\partial_j\divv\mathbf{u})\big]\mathrm{d}\mathbf{x}
-\sigma\int\dot{u}^j\big[\partial_j(B^iB^i_t)+\divv(B^i\partial_jB^i\mathbf{u})\big]
\mathrm{d}\mathbf{x}\notag\\
&\quad+\sigma\int\dot{u}^j\big[\partial_t(B^i\partial_iB^j)
+\divv(B^i\partial_iB^j\mathbf{u})\big]
\mathrm{d}\mathbf{x}
\triangleq \sum_{i=1}^{5}\mathcal{J}_i.
\end{align}

We next estimate each $\mathcal{J}_{i}$. Using Cauchy--Schwarz, Poincar\'{e}'s, and H\"older's inequalities, we have that
\begin{align}\label{3.21}
\mathcal{J}_{1}&=-\sigma\int_{\partial\Omega}P_{t}\dot{\mathbf{u}}\cdot \mathbf{n}\mathrm{d}s+\sigma\int P_{t}\divv\dot{\mathbf{u}}\mathrm{d}\mathbf{x}-\sigma\int\dot{\mathbf{u}}\cdot\nabla\divv(P\mathbf{u})\mathrm{d}\mathbf{x}
+\sigma\int\dot{u}^j\divv(P\partial_{j}\mathbf{u})\mathrm{d}\mathbf{x}\notag\\
&=-\sigma\int_{\partial\Omega}\big(P_{t}+\divv(P\mathbf{u})\big)\dot{\mathbf{u}}\cdot \mathbf{n}\mathrm{d}s
+\sigma\int \big(P_{t}+\divv(P\mathbf{u})\big)\divv\dot{\mathbf{u}}\mathrm{d}\mathbf{x}
+\sigma\int\dot{\mathbf{u}}\cdot\nabla\mathbf{u}\cdot\nabla P\mathrm{d}\mathbf{x}
\notag\\&\quad+\sigma\int P\dot{\mathbf{u}}\cdot\nabla\divv\mathbf{u}\mathrm{d}\mathbf{x}\notag\\
&=-\sigma\int_{\partial\Omega}\big(P_{t}+\divv(P\mathbf{u})\big)\dot{\mathbf{u}}\cdot \mathbf{n}\mathrm{d}s
+\sigma\int_{\partial\Omega} P\dot{\mathbf{u}}\cdot\nabla\mathbf{u}\cdot\mathbf{n}\mathrm{d}s
-\sigma\int(\gamma-1)P\divv\mathbf{u}\divv\dot{\mathbf{u}}\mathrm{d}\mathbf{x}
-\sigma\int P\dot{u}_j^iu_i^j \mathrm{d}\mathbf{x}
\notag\\
&\triangleq\mathcal{B}_1+\mathcal{B}_2
-\sigma(\gamma-1)\int P\divv\mathbf{u}\divv\dot{\mathbf{u}}\mathrm{d}\mathbf{x}
-\sigma\int P\dot{u}_j^iu_i^j \mathrm{d}\mathbf{x}\notag\\
&\leq \mathcal{B}_1+\mathcal{B}_2+\frac{\tilde{C}\sigma}{16}\|\nabla\dot{\mathbf{u}}\|_{L^2}^2+C\sigma\|\nabla\mathbf{u}\|_{L^2}^2,
\end{align}
where the positive constant $\tilde{C}=\tilde{C}(\mu,\Omega)$ will be determined later.
Integration by parts together with \eqref{a3} and \eqref{2.12} gives that
\begin{align}\label{3.22}
\mathcal{J}_2&=-\mu\sigma\int\dot{\mathbf{u}}\cdot\nabla^\bot\curl \mathbf{u}_{t}\mathrm{d}\mathbf{x}
-\mu\sigma\int\dot{u}^j\divv(\mathbf{u}\nabla^\bot_j\curl \mathbf{u})\mathrm{d}\mathbf{x}\notag\\
&=-\mu\sigma\int_{\partial\Omega}\curl\mathbf{u}_t
\dot{\mathbf{u}}\cdot\mathbf{n}^\bot\mathrm{d}s
-\mu\sigma\int\curl\dot{\mathbf{u}}\curl\mathbf{u}_t\mathrm{d}\mathbf{x}
-\mu\sigma\int\dot{u}^j\big[\nabla^\bot_j\divv(\mathbf{u}\curl \mathbf{u})-\divv(\nabla^\bot_j\mathbf{u}\curl \mathbf{u})\big]\mathrm{d}\mathbf{x}\notag\\
&=-\mu\sigma\int_{\partial\Omega}\curl\mathbf{u}_t
\dot{\mathbf{u}}\cdot\mathbf{n}^\bot\mathrm{d}s
-\mu\sigma\int\big[(\curl\dot{\mathbf{u}})^2-\curl\dot{\mathbf{u}}
\curl(\mathbf{u}\cdot\nabla\mathbf{u})\big]\mathrm{d}\mathbf{x}
-\mu\sigma\int_{\partial\Omega}\divv(\mathbf{u}\curl\mathbf{u})
\dot{\mathbf{u}}\cdot\mathbf{n}^\bot\mathrm{d}s
\notag\\
&\quad
-\mu\sigma\int\curl\dot{\mathbf{u}}\divv(\mathbf{u}\curl\mathbf{u})
\mathrm{d}\mathbf{x}
+\mu\sigma\int_{\partial\Omega}(\curl\mathbf{u})\dot{u}^j\nabla^\bot_j
\mathbf{u}\cdot\mathbf{n}\mathrm{d}s
-\mu\sigma\int\nabla\dot{u}^j\cdot\nabla^\bot_j\mathbf{u}\curl\mathbf{u}\mathrm{d}\mathbf{x}\notag\\
&=-\mu\sigma\int(\curl\dot{\mathbf{u}})^2\mathrm{d}\mathbf{x}
-\mu\sigma\int\nabla\dot{u}^j\cdot\nabla^\bot_j\mathbf{u}
\curl\mathbf{u}\mathrm{d}\mathbf{x}\notag\\
&\leq-\mu\sigma\|\curl\dot{\mathbf{u}}\|_{L^2}^2
+\frac{\tilde{C}\sigma}{16}\|\nabla\dot{\mathbf{u}}\|_{L^2}^2
+C\sigma\|\nabla\mathbf{u}\|_{L^4}^4\notag\\
&\leq-\mu\sigma\|\curl\dot{\mathbf{u}}\|_{L^2}^2
+C\sigma\big(1+\|\nabla\mathbf{u}\|_{L^2}^2\big)
\big(\|\sqrt{\rho}\dot{\mathbf{u}}\|_{L^2}^2+\|\nabla\mathbf{u}\|_{L^2}^2\big)
+\frac{C}{(2\mu+\lambda)^4}\|P-\bar{P}\|_{L^4}^4\notag\\
&\quad+
C\sigma(1+C_0)^2\bigg(1+\frac{\|\nabla\mathbf{B}\|_{L^2}^2}{(2\mu+\lambda)^4}\bigg)
\|\nabla\mathbf{B}\|_{L^2}^2\big(1+\|\nabla\mathbf{B}\|_{L^2}^2\big)
\|\nabla\curl\mathbf{B}\|_{L^2}^2,
\end{align}
where in the fourth equality we have used the fact that $\curl(\mathbf{u}\cdot\nabla\mathbf{u})=\divv(\mathbf{u}\curl\mathbf{u})$.
To address $\mathcal{J}_3$, we observe that
\begin{align}\label{3.23}
\mathcal{J}_3
&=(2\mu+\lambda)\sigma\int\dot{u}^j\big[\partial_j\divv\mathbf{u}_t+ \partial_j\divv(\mathbf{u}\divv\mathbf{u})-\divv(\partial_j\mathbf{u}\divv\mathbf{u})\big]
\mathrm{d}\mathbf{x}\notag\\
&=(2\mu+\lambda)\sigma\int_{\partial\Omega}
[\divv\mathbf{u}_t+\divv(\mathbf{u}\divv\mathbf{u})]\dot{\mathbf{u}}\cdot\mathbf{n}\mathrm{d}s
-(2\mu+\lambda)\sigma\int\divv\dot{\mathbf{u}}
[\divv\mathbf{u}_t+\divv(\mathbf{u}\divv\mathbf{u})]\mathrm{d}\mathbf{x}\notag\\
&\quad
-(2\mu+\lambda)\sigma\int\dot{u}^j\divv(\partial_j\mathbf{u}\divv\mathbf{u})\mathrm{d}\mathbf{x}\notag\\
&=-(2\mu+\lambda)\sigma\int\big(\divv\mathbf{u}_t+\mathbf{u}\cdot\nabla\divv\mathbf{u}+u_j^iu_i^j\big)
\big[\divv\mathbf{u}_t+\mathbf{u}\cdot\nabla\divv\mathbf{u}+(\divv\mathbf{u})^2\big]\mathrm{d}\mathbf{x}
\notag\\&\quad
-(2\mu+\lambda)\sigma\int\dot{u}^j\divv(\partial_j\mathbf{u}\divv\mathbf{u})\mathrm{d}\mathbf{x}
+(2\mu+\lambda)\sigma\int_{\partial\Omega}[\divv\mathbf{u}_t+\divv(\mathbf{u}\divv\mathbf{u})]\dot{\mathbf{u}}\cdot\mathbf{n}\mathrm{d}s\notag\\
&=-(2\mu+\lambda)\sigma\int\big[(\divf\mathbf{u})^2
+\divf\mathbf{u}(\divv\mathbf{u})^2+u_j^iu_i^j(\divv\mathbf{u})^2
-\partial_j\mathbf{u}\cdot\nabla\dot{u}^j\divv\mathbf{u}+u_j^iu_i^j\divf\mathbf{u}\big]\mathrm{d}\mathbf{x}\notag\\
&\quad
+(2\mu+\lambda)\sigma\int_{\partial\Omega}\divv\mathbf{u}_t\dot{\mathbf{u}}\cdot\mathbf{n}\mathrm{d}s+
(2\mu+\lambda)\sigma\int_{\partial\Omega}\divv(\mathbf{u}\divv\mathbf{u})\dot{\mathbf{u}}\cdot\mathbf{n}\mathrm{d}s
-(2\mu+\lambda)\sigma\int_{\partial\Omega}(\divv\mathbf{u})\dot{\mathbf{u}}\cdot\nabla\mathbf{u}\cdot\mathbf{n}\mathrm{d}s\notag\\
&\triangleq-(2\mu+\lambda)\sigma\|\divf\mathbf{u}\|_{L^2}^2+\sum_{i=1}^4\mathcal{J}_{3i}+\mathcal{B}_3+\mathcal{B}_4+\mathcal{B}_5.
\end{align}

We shall bound each $\mathcal{J}_{3i}$. It follows from Gagliardo--Nirenberg inequality, H\"older's inequality, Lemma $\ref{E0}$, and Lemma $\ref{l3.1}$ that
\begin{align}\label{3.24}
\mathcal{J}_{31}&=-\frac{\sigma}{2\mu+\lambda}\int\divf\mathbf{u}
\Big(F+P-\bar{P}+\frac12|\mathbf{B}|^2\Big)^2\mathrm{d}\mathbf{x}\notag\\
&\leq\frac{C\sigma}{2\mu+\lambda}\|\divf\mathbf{u}\|_{L^2}
\big(\|F\|_{L^4}^2+\|P-\bar{P}\|_{L^4}^2+\|\mathbf{B}\|_{L^\infty}^2
\|\mathbf{B}\|_{L^4}^2\big)\notag\\
&\leq\frac{C\sigma}{2\mu+\lambda}\|\divf\mathbf{u}\|_{L^2}
\Big[\big((2\mu+\lambda)\|\divv\mathbf{u}\|_{L^2}+\|P-\bar{P}\|_{L^2}\big)
\big(\|\sqrt{\rho}\dot{\mathbf{u}}\|_{L^2}+\|\nabla \mathbf{u}\|_{L^2}+\||\mathbf{B}||\nabla\mathbf{B}|\|_{L^2}\big)
\notag\\
&\quad+\|P-\bar{P}\|_{L^4}^2+C_0^\frac34\|\nabla\mathbf{B}\|_{L^2}^2
\Big(\|\nabla\mathbf{B}\|_{L^2}^\frac12+\|\nabla\curl\mathbf{B}\|_{L^2}^\frac12\Big)\Big]
\notag\\
&\leq\frac{(2\mu+\lambda)\sigma}{16}\|\divf\mathbf{u}\|_{L^2}^2
+C\sigma\big(1+C_0+\|\nabla\mathbf{u}\|_{L^2}^2+C_0\|\nabla\mathbf{B}\|_{L^2}^2\big)
\big(\|\sqrt{\rho}\dot{\mathbf{u}}\|_{L^2}^2+\|\nabla\mathbf{u}\|_{L^2}^2\big)\notag\\
&\quad+C\sigma(1+C_0)^2\|\nabla\mathbf{B}\|_{L^2}^4\|\nabla\curl\mathbf{B}\|_{L^2}^2
+\frac{C}{(2\mu+\lambda)^3}\|P-\bar{P}\|_{L^4}^4.
\end{align}
According to the Hodge-type decomposition, we conclude from Lemma $\ref{E0}$ that
\begin{align}\label{3.25}
\mathcal{J}_{32}&=-(2\mu+\lambda)\sigma\int u_j^iu_i^j(\divv\mathbf{u})^2\mathrm{d}\mathbf{x}\notag\\
&\leq C(2\mu+\lambda)\sigma\int\big(|\nabla\mathcal{P}\mathbf{u}|^2+|\nabla\mathcal{Q}\mathbf{u}|^2\big)(\divv\mathbf{u})^2\mathrm{d}\mathbf{x}\notag\\
&\leq C(2\mu+\lambda)^2\sigma\|\divv\mathbf{u}\|_{L^4}^4+C\sigma\|\nabla\mathcal{P}\mathbf{u}\|_{L^4}^4\notag\\
&\leq\frac{C\sigma}{(2\mu+\lambda)^2}\big(\|F\|_{L^4}^4+\|P-\bar{P}\|_{L^4}^4
+\||\mathbf{B}|^2\|_{L^4}^4\big)
+C\sigma\|\nabla\mathbf{u}\|_{L^{2}}^{2}\big(\|\sqrt{\rho}\dot{\mathbf{u}}\|_{L^{2}}^{2}+\|\nabla\mathbf{u}\|_{L^{2}}^{2}
+\|\mathbf{B}\cdot\nabla\mathbf{B}\|_{L^2}^2\big)\notag\\
&\leq C\sigma\big(1+C_0+\|\nabla\mathbf{u}\|_{L^2}^2+C_0\|\nabla\mathbf{B}\|_{L^2}^2\big)
\big(\|\sqrt{\rho}\dot{\mathbf{u}}\|_{L^2}^2+\|\nabla\mathbf{u}\|_{L^2}^2\big)
+C\sigma(1+C_0)^3\|\nabla\mathbf{B}\|_{L^2}^4\|\nabla\curl\mathbf{B}\|_{L^2}^2\notag\\
&\quad+\frac{C}{(2\mu+\lambda)^2}\|P-\bar{P}\|_{L^4}^4,\\
\mathcal{J}_{33}&=\sigma\int\partial_j\mathbf{u}\cdot\nabla\dot{u}^j
\Big(F+P-\bar{P}+\frac12|\mathbf{B}|^2\Big)\mathrm{d}\mathbf{x}\notag\\
&\leq
C\sigma\big[\|\nabla\mathbf{u}\|_{L^4}\|\nabla\dot{\mathbf{u}}\|_{L^2}
(\|F\|_{L^4}+\|\mathbf{B}\|_{L^\infty}\|\mathbf{B}\|_{L^4})
+\|\nabla\mathbf{u}\|_{L^2}\|\nabla\dot{\mathbf{u}}\|_{L^2}\|P-\bar{P}\|_{L^\infty}\big]\notag\\
&\leq C\sigma\big[\|\nabla\mathbf{u}\|_{L^4}\|\nabla\dot{\mathbf{u}}\|_{L^2}
(\|\sqrt{\rho}\dot{\mathbf{u}}\|_{L^2}+\|\nabla \mathbf{u}\|_{L^2}+\||\mathbf{B}||\nabla\mathbf{B}|\|_{L^2}+\|\mathbf{B}\|_{L^\infty}\|\mathbf{B}\|_{L^4})
+\|\nabla\mathbf{u}\|_{L^2}\|\nabla\dot{\mathbf{u}}\|_{L^2}\big]\notag\\
&\leq
\frac{\tilde{C}\sigma}{16}\|\nabla \dot{\mathbf{u}}\|_{L^2}^2
+C\sigma \big(1+\|\sqrt{\rho}\dot{\mathbf{u}}\|_{L^2}^2+\|\nabla\mathbf{u}\|_{L^2}^2\big)
\big(\|\sqrt{\rho}\dot{\mathbf{u}}\|_{L^2}^2+\|\nabla\mathbf{u}\|_{L^2}^2\big)
+\frac{C}{(2\mu+\lambda)^2}\|P-\bar{P}\|_{L^4}^4\notag\\
&\quad+
C\sigma(1+C_0)^2\bigg(1+\frac{\|\nabla\mathbf{B}\|_{L^2}^2}{(2\mu+\lambda)^4}\bigg)\|\nabla\mathbf{B}\|_{L^2}^2
\big(1+\|\nabla\mathbf{B}\|_{L^2}^2\big)\|\nabla\curl\mathbf{B}\|_{L^2}^2.\label{3.26}
\end{align}
For $\mathcal{J}_{34}$, one gets from the Hodge-type decomposition that
\begin{align}\label{3.27}
\mathcal{J}_{34}
&=-(2\mu+\lambda)\sigma\int\divf\mathbf{u}(\mathcal{Q}\mathbf{u}+\mathcal{P}\mathbf{u})_j^i(\mathcal{Q}\mathbf{u}
+\mathcal{P}\mathbf{u})_i^j\mathrm{d}\mathbf{x}\notag\\
&=-(2\mu+\lambda)\sigma\int\divf\mathbf{u}\big[(\mathcal{Q}\mathbf{u})_j^i(\mathcal{Q}\mathbf{u})_i^j+
(\mathcal{Q}\mathbf{u})_j^i(\mathcal{P}\mathbf{u})_i^j+(\mathcal{P}\mathbf{u})_j^i(\mathcal{Q}\mathbf{u})_i^j
+(\mathcal{P}\mathbf{u})_j^i(\mathcal{P}\mathbf{u})_i^j\big]\mathrm{d}\mathbf{x}\notag\\
&\leq-(2\mu+\lambda)\sigma\int\divf\mathbf{u}(\mathcal{P}\mathbf{u})_j^i(\mathcal{P}\mathbf{u})_i^j\mathrm{d}\mathbf{x}
+\frac{(2\mu+\lambda)\sigma}{16}\|\divf\mathbf{u}\|_{L^2}^2
+C(2\mu+\lambda)\sigma\int|\nabla\mathbf{u}|^2
|\nabla\mathcal{Q}\mathbf{u}|^2\mathrm{d}\mathbf{x}\notag\\
&\leq-(2\mu+\lambda)\sigma\int\divf\mathbf{u}(\mathcal{P}\mathbf{u})_j^i(\mathcal{P}\mathbf{u})_i^j\mathrm{d}\mathbf{x}
+\frac{(2\mu+\lambda)\sigma}{16}\|\divf\mathbf{u}\|_{L^2}^2
+C\sigma\|\nabla\mathbf{u}\|_{L^4}^4
+C(2\mu+\lambda)^2\sigma\|\divv\mathbf{u}\|_{L^4}^4\notag\\
&\leq-\sigma\int F_t(\mathcal{P}\mathbf{u})_j^i(\mathcal{P}\mathbf{u})_i^j\mathrm{d}\mathbf{x}
+\frac{(2\mu+\lambda)\sigma}{16}\|\divf\mathbf{u}\|_{L^2}^2
+C\sigma(1+C_0)^3\|\nabla\mathbf{B}\|_{L^2}^4\big(\|\nabla\curl\mathbf{B}\|_{L^2}^2
+\|\mathbf{B}_t\|_{L^2}^2\big)\notag\\
&\quad
+C\sigma\big(1+C_0+\|\nabla\mathbf{u}\|_{L^{2}}^{2}+C_0\|\nabla\mathbf{B}\|_{L^{2}}^{2}\big)
\big(\|\sqrt{\rho}\dot{\mathbf{u}}\|_{L^{2}}^{2}+\|\nabla\mathbf{u}\|_{L^{2}}^{2}
+\|\nabla\mathbf{B}\|_{L^{2}}^{2}\big)
\big(1+\|\sqrt{\rho}\dot{\mathbf{u}}\|_{L^{2}}^{2}\big)\notag\\
&\quad+\frac{C}{(2\mu+\lambda)^2}\|P-\bar{P}\|_{L^4}^4,
\end{align}
where the final step relies on the following estimate
\begin{align*}
&-(2\mu+\lambda)\sigma\int(\divv\mathbf{u}_t+\mathbf{u}\cdot\nabla\divv\mathbf{u})(\mathcal{P}\mathbf{u})_j^i(\mathcal{P}\mathbf{u})_i^j\mathrm{d}\mathbf{x}\notag\\
&=-\sigma\int\big[(P-\bar{P})_t+\mathbf{u}\cdot\nabla(P-\bar{P})+\mathbf{B}\cdot\mathbf{B}_t\big](\mathcal{P}\mathbf{u})_j^i(\mathcal{P}\mathbf{u})_i^j\mathrm{d}\mathbf{x}
-\sigma\int F_t(\mathcal{P}\mathbf{u})_j^i(\mathcal{P}\mathbf{u})_i^j\mathrm{d}\mathbf{x}\notag\\
&\quad-\sigma\int\mathbf{u}\cdot\nabla\Big(F+\frac12|\mathbf{B}|^2\Big)(\mathcal{P}\mathbf{u})_{j}^{i}(\mathcal{P}\mathbf{u})_{i}^{j}\mathrm{d}\mathbf{x}
\notag\\
&=\sigma\int\big[\gamma P\divv\mathbf{u}-(\gamma-1)\overline{P\divv\mathbf{u}}-\mathbf{B}\cdot\mathbf{B}_t\big](\mathcal{P}\mathbf{u})_j^i(\mathcal{P}\mathbf{u})_i^j\mathrm{d}\mathbf{x}
-\sigma\int\mathbf{u}\cdot\nabla\Big(F+\frac12|\mathbf{B}|^2\Big)
(\mathcal{P}\mathbf{u})_{j}^{i}(\mathcal{P}\mathbf{u})_{i}^{j}\mathrm{d}\mathbf{x}
\notag\\
&\quad
-\sigma\int F_t(\mathcal{P}\mathbf{u})_j^i(\mathcal{P}\mathbf{u})_i^j\mathrm{d}\mathbf{x}\notag\\
&\leq C\sigma\big(\|\nabla\mathcal{P}\mathbf{u}\|_{L^4}^2\|\nabla\mathbf{u}\|_{L^2}
+\|\nabla\mathcal{P}\mathbf{u}\|_{L^8}^2\|\mathbf{B}\|_{L^4}\|\mathbf{B}_t\|_{L^2}\big)
+C\sigma\|\mathbf{u}\|_{L^4}\|\nabla\mathcal{P}\mathbf{u}\|_{L^8}^2\big(\|\nabla F\|_{L^2}+\|\mathbf{B}\|_{L^\infty}\|\nabla \mathbf{B}\|_{L^2}\big)
\notag\\
&\quad
-\sigma\int F_t(\mathcal{P}\mathbf{u})_j^i(\mathcal{P}\mathbf{u})_i^j\mathrm{d}\mathbf{x}\notag\\
&\leq
C\sigma\big(1+C_0+\|\nabla\mathbf{u}\|_{L^{2}}^{2}+C_0\|\nabla\mathbf{B}\|_{L^{2}}^{2}\big)
\big(\|\sqrt{\rho}\dot{\mathbf{u}}\|_{L^{2}}^{2}+\|\nabla\mathbf{u}\|_{L^{2}}^{2}
+\|\nabla\mathbf{B}\|_{L^{2}}^{2}\big)
\big(1+\|\sqrt{\rho}\dot{\mathbf{u}}\|_{L^{2}}^{2}\big)\notag\\
&\quad+C\sigma(1+C_0)^2\|\nabla\mathbf{B}\|_{L^2}^4
\big(\|\nabla\curl\mathbf{B}\|_{L^2}^2+\|\mathbf{B}_t\|_{L^2}^2\big)
-\sigma\int F_t(\mathcal{P}\mathbf{u})_j^i(\mathcal{P}\mathbf{u})_i^j\mathrm{d}\mathbf{x}
\end{align*}
due to \eqref{3.18} and Lemma \ref{E0}. Indeed, the term involving $F_t$ in \eqref{3.27}
can be handled via Poincar\'{e}'s inequality and \eqref{E4} as follows:
\begin{align}\label{3.28}
&-\sigma\int F_t
(\mathcal{P}\mathbf{u})_{j}^{i}(\mathcal{P}\mathbf{u})_{i}^{j}\mathrm{d}\mathbf{x}\notag\\
&=-\frac{\mathrm{d}}{\mathrm{d}t}\int\sigma F
(\mathcal{P}\mathbf{u})_{j}^{i}(\mathcal{P}\mathbf{u})_{i}^{j}\mathrm{d}\mathbf{x}
+\sigma'\int F
(\mathcal{P}\mathbf{u})_{j}^{i}(\mathcal{P}\mathbf{u})_{i}^{j}\mathrm{d}\mathbf{x}
+\sigma\int F
(\mathcal{P}\mathbf{u})_{jt}^{i}(\mathcal{P}\mathbf{u})_{i}^{j}\mathrm{d}\mathbf{x}
+\sigma\int F
(\mathcal{P}\mathbf{u})_{j}^{i}(\mathcal{P}\mathbf{u})_{it}^{j}\mathrm{d}\mathbf{x}\notag\\
&=-\frac{\mathrm{d}}{\mathrm{d}t}\int\sigma F
(\mathcal{P}\mathbf{u})_{j}^{i}(\mathcal{P}\mathbf{u})_{i}^{j}\mathrm{d}\mathbf{x}
+\sigma'\int F
(\mathcal{P}\mathbf{u})_{j}^{i}(\mathcal{P}\mathbf{u})_{i}^{j}\mathrm{d}\mathbf{x}
+\sigma\int F
\big(\mathcal{P}(\dot{\mathbf{u}}-\mathbf{u}\cdot\nabla\mathbf{u})\big)_{j}^{i}(\mathcal{P}\mathbf{u})_{i}^{j}\mathrm{d}\mathbf{x}\notag\\
&\quad+\sigma\int F
(\mathcal{P}\mathbf{u})_{j}^{i}\big(\mathcal{P}
(\dot{\mathbf{u}}-\mathbf{u}\cdot\nabla\mathbf{u})\big)_{i}^{j}\mathrm{d}\mathbf{x}\notag\\
&\leq
C\|F\|_{L^2}\|\nabla\mathcal{P}\mathbf{u}\|_{L^4}^2
+C\sigma\|F\|_{L^4}\|\nabla\mathcal{P}\mathbf{u}\|_{L^4}
\big(\|\nabla\dot{\mathbf{u}}\|_{L^2}
+\|\nabla\mathbf{u}\|_{L^4}^2
+\|\mathbf{u}\|_{L^\infty}\|\nabla^2\mathcal{P}\mathbf{u}\|_{L^2}
\big)\notag\\
&\quad
-\frac{\mathrm{d}}{\mathrm{d}t}\int\sigma F
(\mathcal{P}\mathbf{u})_{j}^{i}(\mathcal{P}\mathbf{u})_{i}^{j}\mathrm{d}\mathbf{x}\notag\\
&\leq-\frac{\mathrm{d}}{\mathrm{d}t}\int\sigma F
(\mathcal{P}\mathbf{u})_{j}^{i}(\mathcal{P}\mathbf{u})_{i}^{j}\mathrm{d}\mathbf{x}
+\frac{\tilde{C}\sigma}{16}\|\nabla \dot{\mathbf{u}}\|_{L^2}^2
+C\sigma(1+C_0)^3\|\nabla\mathbf{B}\|_{L^2}^4\|\nabla\curl\mathbf{B}\|_{L^2}^2
+\frac{C}{(2\mu+\lambda)^2}\|P-\bar{P}\|_{L^4}^4\notag\\
&\quad
+C\big(1+C_0+\|\nabla\mathbf{u}\|_{L^2}^2+C_0\|\nabla\mathbf{B}\|_{L^2}^2\big)
\big(1+\sigma\|\sqrt{\rho}\dot{\mathbf{u}}\|_{L^{2}}^{2}+\|\nabla\mathbf{u}\|_{L^{2}}^{2}\big)
\big(\|\sqrt{\rho}\dot{\mathbf{u}}\|_{L^2}^2+\|\nabla\mathbf{u}\|_{L^2}^2\big),
\end{align}
where one has used $0\leq\sigma, \sigma'\leq1$ for $t>0$.

Thus, substituting \eqref{3.24}--\eqref{3.28} into \eqref{3.23}, we obtain that
\begin{align}\label{3.29}
\mathcal{J}_{3}
&\leq-\frac{\mathrm{d}}{\mathrm{d}t}\int\sigma F
(\mathcal{P}\mathbf{u})_{j}^{i}(\mathcal{P}\mathbf{u})_{i}^{j}\mathrm{d}\mathbf{x}
-\frac{3(2\mu+\lambda)\sigma}{4}\|\divf\mathbf{u}\|_{L^2}^2
+\frac{\tilde{C}\sigma}{4}\|\nabla \dot{\mathbf{u}}\|_{L^2}^2+\mathcal{B}_3+\mathcal{B}_4+\mathcal{B}_5\notag\\
&\quad+C\big(1+C_0+\|\nabla\mathbf{u}\|_{L^2}^2+C_0\|\nabla\mathbf{B}\|_{L^2}^2\big)
\big(1+\sigma\|\sqrt{\rho}\dot{\mathbf{u}}\|_{L^{2}}^{2}+\|\nabla\mathbf{u}\|_{L^{2}}^{2}\big)
\big(\|\sqrt{\rho}\dot{\mathbf{u}}\|_{L^2}^2+\|\nabla\mathbf{u}\|_{L^2}^2
+\|\nabla\mathbf{B}\|_{L^2}^2\big)
\notag\\
&\quad+C\sigma(1+C_0)^3\|\nabla\mathbf{B}\|_{L^2}^4
\big(\|\nabla\curl\mathbf{B}\|_{L^2}^2+\|\mathbf{B}_t\|_{L^2}^2\big)
+\frac{C}{(2\mu+\lambda)^2}\|P-\bar{P}\|_{L^4}^4.
\end{align}
For the last terms $\mathcal{J}_{4}$ and $\mathcal{J}_{5}$, one deduces from Gagliardo--Nirenberg inequality and Lemma \ref{ldot} that
\begin{align}\label{3.30}
 &\,\mathcal{J}_{4}+\mathcal{J}_{5}\notag\\
 &=-\sigma\int\dot{u}^j\big[\partial_j(B^iB^i_t)+\divv(B^i\partial_jB^i\mathbf{u})\big]
\mathrm{d}\mathbf{x}
+\sigma\int\dot{u}^j\big[\partial_t(B^i\partial_iB^j)+\divv\big(B^i\partial_iB^j\mathbf{u}\big)\big]
\mathrm{d}\mathbf{x}
\notag\\
&=\sigma\int\dot{u}^j\big[\partial_t(B^i\partial_iB^j)-\partial_j(B^iB^i_t)\big]\mathrm{d}\mathbf{x}+
\sigma\int\big(\partial_k\dot{u}^jB^i\partial_jB^iu^k
-\partial_k\dot{u}^jB^i\partial_iB^ju^k\big)\mathrm{d}\mathbf{x}
\notag\\
&\leq C\sigma\|\dot{\mathbf{u}}\|_{L^4}\|\mathbf{B}_t\|_{L^4}\|\nabla\mathbf{B}\|_{L^2}
+C\sigma\|\nabla\dot{\mathbf{u}}\|_{L^2}\big(\|\mathbf{B}\|_{L^4}\|\mathbf{B}_t\|_{L^4}+
\|\mathbf{B}\|_{L^8}\|\nabla\mathbf{B}\|_{L^4}\|\mathbf{u}\|_{L^8}\big)
-\sigma\int(\mathbf{B}\cdot\mathbf{B}_t)\dot{\mathbf{u}}\cdot\mathbf{n}\mathrm{d}s\notag\\
&\leq \frac{\tilde{C}\sigma}{16}\|\nabla \dot{\mathbf{u}}\|_{L^2}^2
+\frac{\nu\sigma}{8}\|\curl\mathbf{B}_t\|_{L^2}^2
+C\sigma (1+C_0)\big(1+\|\nabla\mathbf{B}\|_{L^2}^4+\|\nabla\mathbf{u}\|_{L^2}^2\big)
\big(\|\nabla\curl\mathbf{B}\|_{L^2}^2+\|\mathbf{B}_t\|_{L^2}^2\big)\notag\\
&\quad +C\sigma \|\nabla\mathbf{u}\|_{L^2}^2\big(1+\|\nabla\mathbf{u}\|_{L^2}^2\big)+\mathcal{B}_6.
\end{align}

It remains to estimate the boundary integrals $\mathcal{B}_i$. According to \eqref{2.12} and the trace theorem, one has 
\begin{align}\label{3.31}
 &\mathcal{B}_1+\mathcal{B}_3+\mathcal{B}_6\notag \\
& =\sigma\int_{\partial\Omega}
\big[-P_{t}-\divv(P\mathbf{u})+(2\mu+\lambda)\divv\mathbf{u}_t
-\mathbf{B}\cdot\mathbf{B}_t\big]
\dot{\mathbf{u}}\cdot\mathbf{n}\mathrm{d}s\notag\\
&=\sigma\int_{\partial\Omega}(F-\bar{P})_t\dot{\mathbf{u}}\cdot\mathbf{n}\mathrm{d}s
-\sigma\int_{\partial\Omega}\divv(P\mathbf{u})\dot{\mathbf{u}}\cdot\mathbf{n}\mathrm{d}s\notag\\
&=-\sigma\int_{\partial\Omega}(F-\bar{P})_t(\mathbf{u}\cdot\nabla\mathbf{n}\cdot \mathbf{u})\mathrm{d}s
-\sigma\int_{\partial\Omega}\divv(P\mathbf{u})\dot{\mathbf{u}}\cdot\mathbf{n}\mathrm{d}s\notag\\
&=-\frac{\mathrm{d}}{\mathrm{d}t}\int_{\partial\Omega}\sigma (F-\bar{P})(\mathbf{u}\cdot\nabla\mathbf{n}\cdot \mathbf{u})\mathrm{d}s+\sigma'\int_{\partial\Omega}(F-\bar{P})(\mathbf{u}\cdot\nabla \mathbf{n}\cdot \mathbf{u})\mathrm{d}s+\sigma\int_{\partial\Omega}(F-\bar{P})(\dot{\mathbf{u}}\cdot\nabla \mathbf{n}\cdot \mathbf{u})\mathrm{d}s\notag\\
&\quad+\sigma\int_{\partial\Omega}(F-\bar{P})(\mathbf{u}\cdot\nabla \mathbf{n}\cdot \dot{\mathbf{u}})\mathrm{d}s-\sigma\int_{\partial\Omega}(F-\bar{P})(\mathbf{u}\cdot\nabla\mathbf{u}\cdot\nabla \mathbf{n}\cdot \mathbf{u})\mathrm{d}s-\sigma\int_{\partial\Omega}(F-\bar{P})(\mathbf{u}\cdot\nabla \mathbf{n}\cdot (\mathbf{u}\cdot\nabla)\mathbf{u})\mathrm{d}s\notag\\
&\quad-\sigma\int_{\partial\Omega}\divv(P\mathbf{u})\dot{\mathbf{u}}\cdot\mathbf{n}\mathrm{d}s\notag\\
&\leq
-\frac{\mathrm{d}}{\mathrm{d}t}\int_{\partial\Omega}\sigma (F-\bar{P})(\mathbf{u}\cdot\nabla\mathbf{n}\cdot \mathbf{u})\mathrm{d}s
-\sigma\int_{\partial\Omega}\divv(P\mathbf{u})\dot{\mathbf{u}}\cdot\mathbf{n}\mathrm{d}s
+\frac{\tilde{C}\sigma}{16}\|\nabla\dot{\mathbf{u}}\|_{L^2}^2
+\frac{C}{(2\mu+\lambda)^2}\|P-\bar{P}\|_{L^4}^4\notag\\
&\quad
+C\big(1+\|\nabla\mathbf{u}\|_{L^2}^2+C_0\|\nabla\mathbf{B}\|_{L^2}^2\big) \big(1+\sigma\|\sqrt{\rho}\dot{\mathbf{u}}\|_{L^2}^2+\|\nabla\mathbf{u}\|_{L^2}^2\big)\big(\|\sqrt{\rho}\dot{\mathbf{u}}\|_{L^2}^2
+\|\nabla\mathbf{u}\|_{L^2}^2\big)\notag\\
&\quad+
C\sigma(1+C_0)^2\|\nabla\mathbf{B}\|_{L^2}^2
\big(1+\|\nabla\mathbf{B}\|_{L^2}^2\big)\|\nabla\curl\mathbf{B}\|_{L^2}^2,
\end{align}
where we have employed \eqref{E4} and the following estimates
\begin{align*}
 &\left|\int_{\partial\Omega}(F-\bar{P})\big[(\mathbf{u}\cdot\nabla \mathbf{n}\cdot \mathbf{u})+
 (\dot{\mathbf{u}}\cdot\nabla \mathbf{n}\cdot \mathbf{u}+\mathbf{u}\cdot\nabla \mathbf{n}\cdot \dot{\mathbf{u}})\big]\mathrm{d}s\right|
 \leq C(1+\|F\|_{H^1})\big(\|\mathbf{u}\|_{H^1}^2+\|\mathbf{u}\|_{H^1}\|\dot{\mathbf{u}}\|_{H^1}\big),\notag\\
&\left|\int_{\partial\Omega}(F-\bar{P})(\mathbf{u}\cdot\nabla)\mathbf{u}\cdot\nabla \mathbf{n}\cdot \mathbf{u}\mathrm{d}s\right|
=\left|\int_{\partial\Omega}(F-\bar{P})(\mathbf{u}\cdot\mathbf{n}^\bot)(\mathbf{n}^\bot\cdot\nabla)\mathbf{u}\cdot\nabla \mathbf{n}\cdot \mathbf{u}\mathrm{d}s\right|\notag\\&
=\left|\int\nabla^\bot\cdot[(F-\bar{P})(\mathbf{u}\cdot\mathbf{n}^\bot)\nabla\mathbf{u}\cdot\nabla\mathbf{n}\cdot \mathbf{u}]\mathrm{d}\mathbf{x}\right|
=\left|\int\nabla u^i\cdot\nabla^\bot(\nabla_i\mathbf{n}\cdot\mathbf{u}(F-\bar{P})
(\mathbf{u}\cdot\mathbf{n}^\bot))\mathrm{d}\mathbf{x}\right|
\notag\\&\leq C(1+\|F\|_{H^1})\|\nabla\mathbf{u}\|_{L^4}\big(\|\mathbf{u}\|_{L^4}^2
+\|\nabla\mathbf{u}\|_{L^4}\|\mathbf{u}\|_{L^4}+\|\mathbf{u}\|_{L^8}^2\big),\notag\\
&\left|\int_{\partial\Omega}(F-\bar{P})\mathbf{u}\cdot\nabla \mathbf{n}\cdot (\mathbf{u}\cdot\nabla)\mathbf{u}\mathrm{d}s\right|\leq C(1+\|F\|_{H^1})\|\nabla\mathbf{u}\|_{L^4}\big(\|\mathbf{u}\|_{L^4}^2
+\|\nabla\mathbf{u}\|_{L^4}\|\mathbf{u}\|_{L^4}+\|\mathbf{u}\|_{L^8}^2\big).
\end{align*}
In addition, an application of the divergence theorem along with \eqref{E4} and Lemma \ref{ldot} shows that
\begin{align}\label{3.32}
 &\mathcal{B}_2+\mathcal{B}_4+\mathcal{B}_5-\sigma\int_{\partial\Omega}\divv(P\mathbf{u})\dot{\mathbf{u}}\cdot\mathbf{n}\mathrm{d}s\notag\\
&=\sigma\int_{\partial\Omega}[P-(2\mu+\lambda)\divv\mathbf{u}]\dot{\mathbf{u}}\cdot\nabla\mathbf{u}\cdot\mathbf{n}\mathrm{d}s
+\sigma\int_{\partial\Omega}\big[(2\mu+\lambda)\divv(\mathbf{u}\divv\mathbf{u}) -\divv(P\mathbf{u})\big]\dot{\mathbf{u}}\cdot \mathbf{n}\mathrm{d}s\notag\\
&=-\sigma\int_{\partial\Omega}\Big(F-\bar{P}+\frac12|\mathbf{B}|^2\Big)\dot{\mathbf{u}}\cdot\nabla\mathbf{u}\cdot\mathbf{n}\mathrm{d}s
  +\sigma\int_{\partial\Omega}\divv\Big[\mathbf{u}\Big(F-\bar{P}+\frac12|\mathbf{B}|^2\Big)
  \Big]\dot{\mathbf{u}}\cdot\mathbf{n}\mathrm{d}s\notag\\
&=\sigma\int\divv\big[\dot{\mathbf{u}}\divv\big(\mathbf{u}(F-\bar{P})\big)-(F-\bar{P})
\dot{\mathbf{u}}\cdot\nabla\mathbf{u}\big]
  \mathrm{d}\mathbf{x}+\frac{\sigma}{2}\int\divv\big[\dot{\mathbf{u}}
  \divv(\mathbf{u}|\mathbf{B}|^2)-|\mathbf{B}|^2\dot{\mathbf{u}}\cdot\nabla\mathbf{u}\big]
  \mathrm{d}\mathbf{x}\notag\\
 &\leq C\sigma(1+\|F\|_{H^1})\|\nabla\mathbf{u}\|_{L^4}
 \big(\|\nabla\dot{\mathbf{u}}\|_{L^2}+\|\dot{\mathbf{u}}\|_{L^4}\big)
  +C\sigma\|\nabla\dot{\mathbf{u}}\|_{L^2}
\big(\|\nabla F\|_{L^2}\|\mathbf{u}\|_{L^\infty}+\|\nabla \mathbf{B}\|_{L^4}\|\mathbf{B}\|_{L^8}\|\mathbf{u}\|_{L^8}\big)\notag\\
  &\quad+C\sigma\|\dot{\mathbf{u}}\|_{L^4}\big(\|\nabla\curl\mathbf{B}\|_{L^2}
  \|\mathbf{B}\|_{L^8}\|\mathbf{u}\|_{L^8}+\|\mathbf{u}\|_{L^4}\|\nabla \mathbf{B}\|_{L^4}^2+\|\nabla\mathbf{u}\|_{L^4}\|\nabla\mathbf{B}\|_{L^2}
  \|\mathbf{B}\|_{L^\infty}\big)\notag\\
  &\quad
  +C\sigma\|\nabla\dot{\mathbf{u}}\|_{L^2}\|\nabla\mathbf{u}\|_{L^4}\|\mathbf{B}\|_{L^8}^2
 +\sigma\int\Big(F-\bar{P}+\frac12|\mathbf{B}|^2\Big)
\big(\dot{\mathbf{u}}\cdot\nabla\divv\mathbf{u}-\dot{\mathbf{u}}\cdot\nabla\divv\mathbf{u}\big)
  \mathrm{d}\mathbf{x}+\sigma\int\dot{u}^iu^j\partial_i\partial_jF\mathrm{d}\mathbf{x}\notag\\
  &\leq
  C\sigma\big(1+C_0\|\nabla\mathbf{B}\|_{L^2}^2+\|\nabla\mathbf{u}\|_{L^2}^2
  +\|\sqrt{\rho}\dot{\mathbf{u}}\|_{L^2}^2\big)
\big(\|\sqrt{\rho}\dot{\mathbf{u}}\|_{L^2}^2+\|\nabla\mathbf{u}\|_{L^2}^2
+\|\nabla\mathbf{B}\|_{L^2}^2\big)
+\frac{C}{(2\mu+\lambda)^2}\|P-\bar{P}\|_{L^4}^4
\notag\\
&\quad
+\frac{\tilde{C}\sigma}{16}\|\nabla\dot{\mathbf{u}}\|_{L^2}^2
+C\sigma (1+C_0)^3\big(1+\|\nabla\mathbf{B}\|_{L^2}^4+\|\nabla\mathbf{u}\|_{L^2}^4\big)
\|\nabla\curl\mathbf{B}\|_{L^2}^2,
\end{align}
where we have used
 \begin{equation*}
\int\dot{u}^iu^j\partial_i\partial_jF\mathrm{d}\mathbf{x}
=-\int\big(\dot{u}^i_ju^j\partial_iF+\dot{u}^iu^j_j\partial_iF\big)\mathrm{d}\mathbf{x}
\leq C\|\nabla F\|_{L^2}\big(\|\nabla\dot{\mathbf{u}}\|_{L^2}\|\mathbf{u}\|_{L^\infty}
+\|\dot{\mathbf{u}}\|_{L^4}\|\nabla\mathbf{u}\|_{L^4}\big).
 \end{equation*}

Next, we need to consider $\sigma\|\mathbf{B}_t\|_{L^2}^2$ and $\sigma\|\nabla\curl\mathbf{B}\|_{L^2}^2$. To this end, differentiating $\eqref{a1}_3$ with respect to $t$ and multiplying the resultant by $\sigma\mathbf{B}_t$, one gets from integration by parts that
\begin{align}\label{3.33}
 &\,\frac12\frac{\mathrm{d}}{\mathrm{d}t}\big(\sigma\|\mathbf{B}_t\|_{L^{2}}^2\big)
 +\nu\sigma\|\curl\mathbf{B}_t\|_{L^{2}}^2
 -\frac12\sigma'\|\mathbf{B}_t\|_{L^{2}}^2\notag\\
 &=\sigma\int\big(\mathbf{B}_t\cdot\nabla\mathbf{u}-\mathbf{u}\cdot\nabla\mathbf{B}_t-\mathbf{B}_t\divv\mathbf{u}\big)\cdot\mathbf{B}_t\mathrm{d}\mathbf{x}
 +\sigma\int\big(\mathbf{B}\cdot\nabla\dot{\mathbf{u}}-\dot{\mathbf{u}}\cdot\nabla\mathbf{B}
 -\mathbf{B}\divv\dot{\mathbf{u}}\big)\cdot\mathbf{B}_t\mathrm{d}\mathbf{x}
 \notag\\&\quad
 -\sigma\int\big(\mathbf{B}\cdot\nabla(\mathbf{u}\cdot\nabla\mathbf{u})
 -(\mathbf{u}\cdot\nabla\mathbf{u})\cdot\nabla\mathbf{B}-\mathbf{B}\divv(\mathbf{u}\cdot\nabla\mathbf{u})\big)
 \cdot\mathbf{B}_t\mathrm{d}\mathbf{x}
 \triangleq \mathcal{L}_1+\mathcal{L}_2+\mathcal{L}_3.
\end{align}
By similar arguments, it can be deduced that
\begin{align}\label{3.34}
\mathcal{L}_1&\leq C\sigma\|\mathbf{B}_t\|_{L^4}^2\|\nabla\mathbf{u}\|_{L^2}\leq
\frac{\nu\sigma}{8}\|\curl\mathbf{B}_t\|_{L^2}^2
+C\sigma\|\nabla\mathbf{u}\|_{L^2}^2\|\mathbf{B}_t\|_{L^2}^2,\\
\mathcal{L}_2&\leq C\sigma
\|\mathbf{B}_t\|_{L^4}\big(\|\nabla\dot{\mathbf{u}}\|_{L^2}\|\mathbf{B}\|_{L^4}
+\|\dot{\mathbf{u}}\|_{L^4}\|\nabla\mathbf{B}\|_{L^2}\big)\notag\\
&\leq \frac{\tilde{C}\sigma}{16}\|\nabla\dot{\mathbf{u}}\|_{L^2}^2
+\frac{\nu\sigma}{8}\|\curl\mathbf{B}_t\|_{L^2}^2
+C\sigma\big(\|\nabla\mathbf{u}\|_{L^2}^4+\|\nabla\mathbf{B}\|_{L^2}^4\big)
\|\mathbf{B}_t\|_{L^2}^2+C\sigma\|\nabla\mathbf{u}\|_{L^2}^4,\label{3.35}\\
\mathcal{L}_3
&\leq C\sigma\big[\|\mathbf{B}\|_{L^4}\|\mathbf{B}_t\|_{L^4}\|\nabla\mathbf{u}\|_{L^4}^2
+\|\mathbf{u}\|_{L^\infty}\|\nabla\mathbf{u}\|_{L^4}
\big(\|\mathbf{B}_t\|_{L^4}\|\nabla\mathbf{B}\|_{L^2}
+\|\mathbf{B}\|_{L^4}\|\nabla\mathbf{B}_t\|_{L^2}\big)\big]\notag\\
&\quad-\sigma\int_{\partial\Omega}(\mathbf{B}\cdot\mathbf{B}_t)\mathbf{u}\cdot \nabla\mathbf{n}\cdot\mathbf{u}\mathrm{d}s\notag\\
&\leq C\sigma\|\mathbf{B}\|_{L^4}\|\mathbf{B}_t\|_{L^4}\|\nabla\mathbf{u}\|_{L^4}^2
+C\sigma\|\mathbf{u}\|_{L^\infty}\|\nabla\mathbf{u}\|_{L^4}
\big(\|\mathbf{B}_t\|_{L^4}\|\nabla\mathbf{B}\|_{L^2}
+\|\mathbf{B}\|_{L^4}\|\nabla\mathbf{B}_t\|_{L^2}\big)\notag\\
&\quad+C\sigma\|\nabla\mathbf{B}\|_{L^2}(\|\mathbf{B}_t\|_{L^2}
+\|\curl\mathbf{B}_t\|_{L^2})\|\nabla\mathbf{u}\|_{L^2}^2\notag\\
&\leq\frac{\nu\sigma}{8}\|\curl\mathbf{B}_t\|_{L^2}^2
+C\sigma\big(1+\|\nabla\mathbf{B}\|_{L^2}^2+\|\nabla\mathbf{u}\|_{L^2}^2\big)
\big(\|\sqrt{\rho}\dot{\mathbf{u}}\|_{L^2}^2+\|\nabla\mathbf{u}\|_{L^2}^2
+\|\nabla\mathbf{B}\|_{L^2}^2\big)
\big(1+\|\nabla\mathbf{B}\|_{L^2}^2\big)\notag\\
&\quad
+C\sigma (1+C_0)^2\big(1+\|\nabla\mathbf{B}\|_{L^2}^4\big)
\big(\|\nabla\curl\mathbf{B}\|_{L^2}^2+\|\mathbf{B}_t\|_{L^2}^2\big)
+\frac{C}{(2\mu+\lambda)^4}\|P-\bar{P}\|_{L^4}^4.\label{3.36}
\end{align}
Moreover, the $L^2$-estimate for $\eqref{a1}_3$ along with the boundary conditions \eqref{a3} implies that
\begin{align}\label{3.37}
  \|\nabla\curl\mathbf{B}\|_{L^2}^2&\leq C(\nu)\big(\|\mathbf{B}_t\|_{L^2}^2+\||\mathbf{u}||\nabla\mathbf{B}|\|_{L^2}^2+\||\mathbf{B}||\nabla\mathbf{u}|\|_{L^2}^2\big)\notag\\
&\leq C\big(\|\mathbf{B}_t\|_{L^2}^2+\|\mathbf{u}\|_{L^4}^2\|\nabla\mathbf{B}\|_{L^4}^2
+\|\mathbf{B}\|_{L^\infty}^2\|\nabla\mathbf{u}\|_{L^2}^2\big)\notag\\
&\leq \frac{1}{2}\|\nabla\curl\mathbf{B}\|_{L^2}^2
+C\|\mathbf{B}_t\|_{L^2}^2
+C\|\nabla\mathbf{u}\|_{L^2}^2\big(1+\|\nabla\mathbf{u}\|_{L^2}^2\big)
\big(1+\|\nabla\mathbf{B}\|_{L^2}^4\big).
\end{align}

Consequently, substituting \eqref{3.21}, \eqref{3.22}, and \eqref{3.29}--\eqref{3.32} into \eqref{3.20}, and adding \eqref{3.33}, one infers from Lemma \ref{l3.2} and
\eqref{3.34}--\eqref{3.37} that
\begin{align}\label{3.38}
&\frac{\mathrm{d}}{\mathrm{d}t}\bigg(\frac{\sigma}{2}\|\sqrt{\rho}\dot{\mathbf{u}}\|_{L^{2}}^{2}
+\frac{\sigma}{2}\|\mathbf{B}_t\|_{L^{2}}^2
+\int\sigma F
(\mathcal{P}\mathbf{u})_{j}^{i}(\mathcal{P}\mathbf{u})_{i}^{j}\mathrm{d}\mathbf{x}
+\int_{\partial\Omega}\sigma (F-\bar{P})(\mathbf{u}\cdot\nabla\mathbf{n}\cdot \mathbf{u})\mathrm{d}s\bigg)
\notag\\&\quad
+\frac{(2\mu+\lambda)\sigma}{4}\|\divf\mathbf{u}\|_{L^2}^2
+\frac{\mu\sigma}{4}\|\curl\dot{\mathbf{u}}\|_{L^2}^2
+\frac{\nu\sigma}{4}\|\curl\mathbf{B}_t\|_{L^{2}}^2\notag\\
&\leq(2+M)^{\exp\big\{3D_2(1+C_0)^3\big\}}\big(1+\sigma\|\sqrt{\rho}\dot{\mathbf{u}}\|_{L^{2}}^{2}
+\|\nabla\mathbf{u}\|_{L^{2}}^{2}+\|\curl\mathbf{B}\|_{L^{2}}^{2}\big)
\big(\|\sqrt{\rho}\dot{\mathbf{u}}\|_{L^{2}}^{2}
+\|\nabla\mathbf{u}\|_{L^{2}}^{2}+\|\curl\mathbf{B}\|_{L^{2}}^{2}\big)\notag\\
&\quad
+(2+M)^{\exp\big\{3D_2(1+C_0)^3\big\}}\bigg[ \big(1+\sigma\|\mathbf{B}_t\|_{L^2}^2\big)\|\mathbf{B}_t\|_{L^2}^2
+\frac{C}{(2\mu+\lambda)^2}\|P-\bar{P}\|_{L^4}^4\bigg],
\end{align}
where, due to Lemma \ref{ldot}, we have chosen the constant $\tilde{C}$ such that
\begin{equation*}
\tilde{C}\|\nabla\dot{\mathbf{u}}\|_{L^2}^2\leq
C_1\tilde{C}\big(\|\divv\dot{\mathbf{u}}\|_{L^2}^2
+\|\curl\dot{\mathbf{u}}\|_{L^2}^2+\|\nabla\mathbf{u}\|_{L^4}^4\big)
\leq \mu\big(\|\divf\mathbf{u}\|_{L^2}^2+\|\curl\dot{\mathbf{u}}\|_{L^2}^2\big)   +C\|\nabla\mathbf{u}\|_{L^4}^4.
\end{equation*}
Moreover, \eqref{E4} and \eqref{3.37} ensure that
\begin{align*}
&\left|\int\sigma F(\mathcal{P}\mathbf{u})_j^i(\mathcal{P}\mathbf{u})_i^j\mathrm{d}\mathbf{x}\right|
+\left|\int_{\partial\Omega}\sigma (F-\bar{P})(\mathbf{u}\cdot\nabla\mathbf{n}\cdot \mathbf{u})\mathrm{d}s\right| \\
&\leq C\sigma(1+\|F\|_{H^1})
\big(\|\nabla\mathcal{P}\mathbf{u}\|_{L^3}^2+\|\nabla\mathbf{u}\|_{L^2}^2\big)\notag\\
&\leq
 \frac{\sigma}{4}\|\sqrt{\rho}\dot{\mathbf{u}}\|_{L^{2}}^2+\frac{\sigma}{4}\|\mathbf{B}_t\|_{L^{2}}^2+
 C\|\nabla\mathbf{u}\|_{L^{2}}^2\big(1+\|\nabla\mathbf{u}\|_{L^{2}}^2\big)^3
 +C(1+C_0)\|\curl\mathbf{B}\|_{L^{2}}^2\big(1+\|\curl\mathbf{B}\|_{L^{2}}^2\big).
\end{align*}
Now we define an auxiliary functional $\mathcal{E}_2(t)$ as
\begin{align*}
\mathcal{E}_2(t)\triangleq&\frac{\sigma}{2}\|\sqrt{\rho}\dot{\mathbf{u}}\|_{L^{2}}^2
+\frac{\sigma}{2}\|\mathbf{B}_t\|_{L^{2}}^2
+(2+M)^{\exp\big\{3D_2(1+C_0)^3\big\}}\mathcal{E}_1(t)\notag\\
&+\int\sigma F
(\mathcal{P}\mathbf{u})_{j}^{i}(\mathcal{P}\mathbf{u})_{i}^{j}\mathrm{d}\mathbf{x}
+\int_{\partial\Omega}\sigma (F-\bar{P})(\mathbf{u}\cdot\nabla\mathbf{n}\cdot \mathbf{u})\mathrm{d}s.
\end{align*}
Then, by the definition of $\mathcal{E}_1(t)$ in \eqref{3.14} and \eqref{3.15}, one sees that
\begin{equation}\label{3.39}
\mathcal{E}_2(t)\thicksim \sigma\|\sqrt{\rho}\dot{\mathbf{u}}\|_{L^{2}}^2
+\sigma\|\mathbf{B}_t\|_{L^{2}}^2+\mathcal{E}_1(t).
\end{equation}

Setting
\begin{equation*}
\begin{cases}
f_2(t)\triangleq2+\mathcal{E}_2(t),\\
g_2(t)\triangleq(2+M)^{\exp\big\{\frac{13}{4}D_2(1+C_0)^3\big\}}
\bigg(\|\sqrt{\rho}\dot{\mathbf{u}}\|_{L^{2}}^2+\|\nabla\mathbf{u}\|_{L^{2}}^2
+\|\mathbf{B}_t\|_{L^{2}}^2+\|\curl\mathbf{B}\|_{L^{2}}^2
+\frac{\|P-\bar{P}\|_{L^4}^4}{(2\mu+\lambda)^2}\bigg),
\end{cases}
\end{equation*}
one gets from \eqref{3.39} and Lemma \ref{Hodge} that
\begin{align*}
f'_2(t)\leq g_2(t)f_2(t)
\end{align*}
provided that $\lambda$ satisfies \eqref{lam} with $D\geq 3D_2$.
Thus, it follows from Gronwall's inequality and Lemmas \ref{l3.1}--\ref{l3.3} that
\begin{equation}\label{3.40}
  \sup_{0\leq t\leq T}\big(\sigma\|\sqrt{\rho}\dot{\mathbf{u}}\|_{L^{2}}^2
  +\sigma\|\mathbf{B}_t\|_{L^{2}}^2\big)\leq
  \exp\bigg\{(2+M)^{\exp\big\{\frac{13}{4}D_2(1+C_0)^3\big\}}\bigg\},
\end{equation}
which combined with \eqref{3.37} implies that
\begin{equation}\label{3.41}
  \sup_{0\leq t\leq T}\big(\sigma\|\nabla\curl\mathbf{B}\|_{L^{2}}^2\big)\leq
  \exp\bigg\{(2+M)^{\exp\big\{\frac{7}{2}D_2(1+C_0)^3\big\}}\bigg\}.
\end{equation}
Integrating \eqref{3.38} with respect to $t$ over $(0,T)$, one obtains from \eqref{3.40} and Lemmas \ref{l3.1}--\ref{l3.3} that
\begin{align*}
&\int_0^T\big[(2\mu+\lambda)\sigma\|\divf\mathbf{u}\|_{L^2}^2+
\mu\sigma\|\curl\dot{\mathbf{u}}\|_{L^2}^2
+\nu\sigma\|\curl\mathbf{B}_t\|_{L^{2}}^2\big]\mathrm{d}t\notag\\
&\leq C(1+M)+C(1+C_0)(2+M)^{\exp\big\{\frac{7}{2}D_2(1+C_0)^3\big\}}
\exp\bigg\{(2+M)^{\exp\big\{\frac{7}{2}D_2(1+C_0)^3\big\}}\bigg\}
\notag\\
&\leq
 \exp\bigg\{(2+M)^{\exp\big\{\frac{15}{4}D_2(1+C_0)^3\big\}}\bigg\},
\end{align*}
which along with \eqref{3.40} and \eqref{3.41} leads to \eqref{3.19}.
\end{proof}

Finally, inspired by \cite{DE97}, we establish the upper bound of density.
\begin{lemma}\label{l3.5}
Under the assumption \eqref{3.1}, it holds that
\begin{align*}
0\leq\rho(\mathbf{x},t)\leq\frac{7}{4}\hat{\rho}~\textit{a.e.}~\mathrm{on}~\Omega\times[0,T]
\end{align*}
provided that $\lambda$ satisfies \eqref{lam} with $D\geq 5D_2$.
\end{lemma}
\begin{proof}
Let $\mathbf{y}\in\Omega$ and define the corresponding particle path $\mathbf{x}(t)$ by
\begin{align*}
\begin{cases}
\mathbf{\dot{x}}(t,\mathbf{y})=\mathbf{u}(\mathbf{x}(t,\mathbf{y}),\mathbf{y)},\\
\mathbf{\dot{x}}(t_0,\mathbf{y})=\mathbf{y}.
\end{cases}
\end{align*}
Assume that there exists $t_1\leq T$ satisfying $\rho(\mathbf{x}(t_1), t_1) = \frac{7}{4}\hat{\rho}$, we take a minimal value of $t_1$ and then choose a maximal value of $t_0<t_1$ such that $\rho(\mathbf{x}(t_0), t_0)=\frac{3}{2}\hat{\rho}$. Thus, $\rho(\mathbf{x}(t),t)\in\big[\frac{3}{2}\hat{\rho},\frac{7}{4}\hat{\rho}\big]$ for $t\in[t_0,t_1]$. We divide the argument into two cases.

\textbf{Case 1:} $t_0<t_1\leq1$. According to \eqref{1.7} and $\eqref{a1}_1$, one has that
\begin{equation*}
(2\mu+\lambda)\frac{\mathrm{d}}{\mathrm{d}t}\ln\rho(\mathbf{x}(t),t)
+P(\rho(\mathbf{x}(t),t))-\bar{P}+\frac12|\mathbf{B}(\mathbf{x}(t),t)|^2=-F(\mathbf{x}(t),t),
\end{equation*}
where $\frac{\mathrm{d}\rho}{\mathrm{d}t}\triangleq \rho_t+\mathbf{u}\cdot\nabla\rho$. Integrating the above equality from $t_0$ to $t_1$ and abbreviating $\rho(\mathbf{x},t)$ by $\rho(t)$ for
convenience, we have that
\begin{equation}\label{3.42}
\ln\rho(\tau)\big|_{t_0}^{t_1}+\frac{1}{2\mu+\lambda}\int_{t_0}^{t_1}\big[P(\rho(\tau))
-\bar{P}\big]\mathrm{d}\tau=-\frac{1}{2\mu+\lambda}\int_{t_0}^{t_1}G(\mathbf{x}(\tau),\tau)\mathrm{d}\tau,
\end{equation}
where
\begin{equation*}
  G= F+\frac12|\mathbf{B}|^2.
\end{equation*}
It follows from Gagliardo--Nirenberg inequality,
Lemmas \ref{E0}, \ref{ldot}, \ref{l3.1}--\ref{l3.4}, and \eqref{2.10} that
\begin{align}\label{3.43}
&\int_0^{\sigma(T)}\|G(\cdot,t)\|_{L^\infty}\mathrm{d}t\leq C\int_0^{\sigma(T)}\|G\|_{L^2}^{\frac13}\|\nabla G\|_{L^4}^{\frac23}\mathrm{d}t\notag\\
&\leq C\int_0^{\sigma(T)}\Big((2\mu+\lambda)^{\frac13}
\|\divv\mathbf{u}\|_{L^2}^{\frac13}+\|P-\bar{P}\|_{L^2}^{\frac13}
\Big)\Big(\|\sqrt{\rho}\dot{\mathbf{u}}\|_{L^4}^{\frac23}+\|\nabla\mathbf{u}\|_{L^4}^{\frac23}
+\||\mathbf{B}||\nabla\mathbf{B}|\|_{L^4}^{\frac23}\Big)
\mathrm{d}t\notag\\
&\leq C\sup_{0\leq t\leq T}\Big((2\mu+\lambda)^{\frac13}\|\divv\mathbf{u}\|_{L^2}^{\frac13}
+\|P-\bar{P}\|_{L^2}^{\frac13}\Big)
\int_{0}^{\sigma(T)}\Big(
\|\nabla\dot{\mathbf{u}}\|_{L^{2}}^{\frac{2}{3}}
+\|\nabla\mathbf{u}\|_{L^2}^{\frac43}+\|\nabla\mathbf{u}\|_{L^4}^{\frac23}\notag\\
&\quad
+C_0^\frac{1}{12}\big(1+\|\nabla\mathbf{B}\|_{L^2}^2\big)
\|\nabla\curl\mathbf{B}\|_{L^2}^{\frac12}\Big)
\mathrm{d}t\notag\\
&\leq C\Big[(2\mu+\lambda)^{\frac{1}{6}}(2+M)^{\frac16\exp\big\{3D_2(1+C_0)^3\big\}}+1\Big]
\notag \\
& \quad \times\int_{0}^{\sigma(T)}
\Big(\|\nabla\dot{\mathbf{u}}\|_{L^{2}}^{\frac23}
+\|\sqrt{\rho}\dot{\mathbf{u}}\|_{L^{2}}^{\frac23}
+C_0^\frac{1}{12}\|\nabla\mathbf{B}\|_{L^2}^2\|\nabla\curl\mathbf{B}\|_{L^2}^{\frac12}\Big)
\mathrm{d}t\notag\\
&\leq(2\mu+\lambda)^{\frac{1}{6}}(2+M)^{\frac12\exp\big\{3D_2(1+C_0)^3\big\}}
\Bigg[\bigg(\int_{0}^{\sigma(T)}\|\sqrt{\rho}\dot{\mathbf{u}}\|_{L^{2}}^{2}\mathrm{d}t\bigg)^{\frac{1}{3}}\bigg(\int_{0}^{\sigma(T)}1\mathrm{d}t\bigg)^{\frac{2}{3}}
+\bigg(\int_0^{\sigma(T)}
t^{-\frac12}\mathrm{d}t\bigg)^{\frac23}
\notag\\
&\quad\times\bigg(\int_0^{\sigma(T)}\big(t\|\divf{\mathbf{u}}\|_{L^2}^2
+t\|\curl\dot{\mathbf{u}}\|_{L^2}^2\big)
\mathrm{d}t\bigg)^{\frac13}+\sup_{0\leq t\leq T}\|\nabla\mathbf{B}\|_{L^2}^2
\bigg(\int_{0}^{\sigma(T)}\|\nabla\curl\mathbf{B}\|_{L^{2}}^{2}\mathrm{d}t\bigg)^{\frac{1}{4}}\bigg(\int_{0}^{\sigma(T)}1\mathrm{d}t\bigg)^{\frac{3}{4}}
\Bigg]\notag\\
&\leq(2\mu+\lambda)^{\frac{1}{6}}\exp
\bigg\{(2+M)^{\exp\big\{\frac{17}{4}D_2(1+C_0)^3\big\}}\bigg\},
\end{align}
where we have used
\begin{equation*}
  \|\nabla\dot{\mathbf{u}}\|_{L^2}\leq C\big(\|\divf\mathbf{u}\|_{L^2}+\|\curl \dot{\mathbf{u}}\|_{L^2}+\|\nabla \mathbf{u}\|_{L^4}^2\big)
\end{equation*}
due to Lemma $\ref{ldot}$ and the definition of $\divf\mathbf{u}$.
Note that $\rho(t)$ takes values in $\big[\frac{3}{2}\hat{\rho},\frac{7}{4}\hat{\rho}\big]\subset [\hat{\rho},2\hat{\rho}]$ and $P(\rho)$ is increasing on $[0,\infty)$. Substituting $\eqref{3.43}$ into $\eqref{3.42}$ yields
\begin{equation*}
    \ln\left(\frac{7}{4}\hat{\rho}\right)-\ln\left(\frac{3}{2}\hat{\rho}\right)
    \leq\frac{1}{\left(2\mu+\lambda\right)^{\frac{5}{6}}}
   \exp
\bigg\{(2+M)^{\exp\big\{\frac{9}{2}D_2(1+C_0)^3\big\}}\bigg\},
\end{equation*}
which is impossible if $\lambda$ satisfies \eqref{lam} with $D\geq 5D_2$. Therefore, there is no time $t_1$ such that $\rho(\mathbf{x}(t_1), t_1) = \frac{7}{4}\hat{\rho}$. Since $\mathbf{y}\in\Omega$ is arbitrary, it follows that $\rho<\frac{7}{4}\hat{\rho}$ \textit{a.e.} on $\Omega\times[0,T]$.

\textbf{Case 2:} $t_1>1$. From $\eqref{a1}_1$ and \eqref{1.7} we have that
\begin{equation*}
    \frac{\mathrm{d}}{\mathrm{d}t}\rho(t)+\frac{1}{2\mu+\lambda}\rho(t)\Big(P(\rho(t))-\bar{P}+\frac12|\mathbf{B}(\mathbf{x}(t),t)|^2\Big)=-\frac{1}{2\mu+\lambda}\rho(t)F(\mathbf{x}(t),t).
\end{equation*}
After multiplying the above equality by $|\rho(t)|\rho(t)$, one has that
\begin{equation}\label{3.44}
    \frac{1}{3}\frac{\mathrm{d}}{\mathrm{d}t}\left|\rho(t)\right|^3+\frac{1}{2\mu+\lambda}|\rho(t)|^3(P(\rho(t))-\bar{P})
    =-\frac{1}{2\mu+\lambda}|\rho(t)|^3G(\mathbf{x}(t),t).
\end{equation}
If $\rho(t)$ takes values in $[\frac{3}{2}\hat{\rho},\frac{7}{4}\hat{\rho}]$,
by integrating \eqref{3.44} from $t_0$ to $t_1$, we obtain from similar arguments that
\begin{align*}
\hat{\rho}^{3}
&\leq\frac{C}{2\mu+\lambda}\int_{0}^{1}
\|G(\cdot,t)\|_{L^\infty}\mathrm{d}t
+\frac{C}{2\mu+\lambda}\int_{1}^{T}
\|G(\cdot,t)\|_{L^\infty}^3\mathrm{d}t\notag\\
&\leq(2\mu+\lambda)^{-\frac{5}{6}}\exp
\bigg\{(2+M)^{\exp\big\{\frac{17}{4}D_2(1+C_0)^3\big\}}\bigg\}
+\frac{C}{2\mu+\lambda}\int_{1}^{T}\|G\|_{L^{2}}\|\nabla G\|_{L^{4}}^{2}\mathrm{d}t\notag\\
&\leq(2\mu+\lambda)^{-\frac{5}{6}}\exp
\bigg\{(2+M)^{\exp\big\{\frac{17}{4}D_2(1+C_0)^3\big\}}\bigg\}
\notag\\&\quad
+(2\mu+\lambda)^{-\frac{1}{2}}\exp
\bigg\{(2+M)^{\exp\big\{4D_2(1+C_0)^3\big\}}\bigg\}
\int_{0}^{T}
\big(\|\sqrt{\rho}\dot{\mathbf{u}}\|_{L^{2}}^2+\|\nabla\dot{\mathbf{u}}\|_{L^{2}}^2
+C_0\|\nabla\mathbf{B}\|_{L^2}^{4}
\|\nabla\curl\mathbf{B}\|_{L^2}^{2}\big)\mathrm{d}t\notag\\
&\leq\frac{1}{(2\mu+\lambda)^{\frac12}}
\exp
\bigg\{(2+M)^{\exp\big\{\frac{19}{4}D_2(1+C_0)^3\big\}}\bigg\},
\end{align*}
which is impossible if $\lambda$ satisfies \eqref{lam} with $D\geq 5D_2$. Hence we conclude that there is no time $t_1$ such that $\rho(\mathbf{x}(t_1), t_1) = \frac{7}{4}\hat{\rho}$. Since $\mathbf{y}\in\Omega$ is arbitrary, it follows that $\rho<\frac{7}{4}\hat{\rho}$ \textit{a.e.} on $\Omega\times[0,T]$.
\end{proof}

Now we are ready to prove Proposition $\ref{p3.1}$.
\begin{proof}[Proof of Proposition \ref{p3.1}.]
Proposition \ref{p3.1} follows from Lemmas $\ref{l3.2}$--$\ref{l3.5}$ provided that $\lambda$ satisfies \eqref{lam} with $D\geq 5D_2$.
\end{proof}

\section{Proof of Theorem \ref{t1.1}}\label{sec4}
This section employs the \textit{a priori} estimates derived in Section $\ref{sec3}$ to finalize the proof of Theorem \ref{t1.1}.
\begin{proof}[Proof of Theorem \ref{t1.1}.]
Let $(\rho_0, \mathbf{u}_0, \mathbf{B}_0)$ be the initial data as described in Theorem \ref{t1.1}.
For $\epsilon>0$, let $j_\epsilon=j_\epsilon(\mathbf{x})$ be the standard mollifier,
and define the approximate initial data $(\rho_0^\epsilon, \mathbf{u}_0^\epsilon, \mathbf{B}_0^\epsilon)$:
\begin{align*}
\rho_0^\epsilon&=[J_\epsilon\ast(\rho_0\mathbf{1}_\Omega)]\mathbf{1}_\Omega+\epsilon,
\end{align*}
and $(\mathbf{u}_0^\epsilon,\mathbf{B}_0^\epsilon)$ is the unique smooth solution to the following elliptic equations
\begin{equation*}
\begin{cases}
\Delta \mathbf{u}_0^\epsilon = \Delta(J_\epsilon\ast\mathbf{u}_0), & \mathbf{x}\in \Omega, \\
\mathbf{u}_0^\epsilon \cdot \mathbf{n} = 0, \,\ \curl \mathbf{u}_0^\epsilon = 0, & \mathbf{x}\in\partial \Omega,
\end{cases}~~~~~~~~~~~
\begin{cases}
\Delta \mathbf{B}_0^\epsilon = \Delta(J_\epsilon\ast\mathbf{B}_0), & \mathbf{x}\in \Omega, \\
\mathbf{B}_0^\epsilon \cdot \mathbf{n} = 0, \,\ \curl \mathbf{B}_0^\epsilon = 0, & \mathbf{x}\in\partial \Omega.
\end{cases}
\end{equation*}
Then we have
\begin{align*}
\rho_0^\epsilon\in H^2,\ \ (\mathbf{u}_0^\epsilon,\mathbf{B}_0^\epsilon)\in H_\omega^2,\ \
\text{and} \ \ \inf_{\mathbf{x}\in\Omega}\{\rho_0^\epsilon(\mathbf{x})\}\geq\epsilon.
\end{align*}
Proposition $\ref{p3.1}$ implies that, for $\epsilon$ being suitably small,
\begin{align*}
0\leq\rho^\epsilon(\mathbf{x},t)\leq\frac{7}{4}\hat{\rho}~\textit{a.e.}~\mathrm{on}~\Omega\times[0,T]
\end{align*}
provided that $\lambda$ satisfies \eqref{lam}. Thus, Lemma $\ref{l2.1}$ yields the global existence and uniqueness of strong solutions $(\rho^\epsilon,\mathbf{u}^\epsilon,\mathbf{B}^\epsilon)$ to the problem \eqref{a1} and \eqref{a3} with the initial data $(\rho_0^\epsilon,\mathbf{u}_0^\epsilon,\mathbf{B}_0^\epsilon)$.

Fix $\mathbf{x}\in\overline{\Omega}$ and let $B_R$ be a ball of radius $R$ centered at $\mathbf{x}$.
Let $(F^\epsilon,\omega^\epsilon)$ be the functions $(F,\omega)$ with $(\rho,{\bf u},{\bf B})$ replaced by $(\rho^\epsilon,{\bf u}^\epsilon,{\bf B}^\epsilon)$.
Then, for $t\geq\tau>0$, one gets from Lemmas $\ref{E0}$, $\ref{l3.1}\text{--}\ref{l3.4}$, and Sobolev's inequality that
\begin{align}\label{4.1}
\langle\mathbf{u}^\epsilon(\cdot,t)
\rangle^{\frac12}_{\overline{\Omega}}&\leq C(1+\|\nabla\mathbf{u}^\epsilon\|_{L^4})\notag\\
&\leq C\Big(\|\sqrt{\rho^\epsilon}\dot{\mathbf{u}}^\epsilon\|_{L^2}^{\frac{1}{2}}
+\||\mathbf{B}^\epsilon||\nabla\mathbf{B}^\epsilon|\|_{L^2}^{\frac{1}{2}}\Big)
\|\nabla\mathbf{u}^\epsilon\|_{L^2}^\frac{1}{2}+C\|\nabla\mathbf{u}^\epsilon\|_{L^{2}}
+\frac{C}{2\mu+\lambda}\|P(\rho^\epsilon)-\overline{P(\rho^\epsilon)}\|_{L^4}
\notag\\
&\quad
+\frac{C}{2\mu+\lambda}\Big[\|\mathbf{B}^\epsilon\|_{L^{8}}^2
+\|P(\rho^\epsilon)-\overline{P(\rho^\epsilon)}\|_{L^2}^{\frac{1}{2}}
\Big(
\|\sqrt{\rho^\epsilon}\dot{\mathbf{u}}^\epsilon\|_{L^2}^{\frac{1}{2}}
+\||\mathbf{B}^\epsilon||\nabla\mathbf{B}^\epsilon|\|_{L^2}^{\frac{1}{2}}
+\|\nabla\mathbf{u}^\epsilon\|_{L^{2}}^\frac12\Big)
\Big]+C
\notag\\
&\leq C(\tau).
\end{align}
Note that
\begin{align*}
\left|\mathbf{u}^\epsilon(\mathbf{x},t)-\frac{1}{|B_R\cap\Omega|}
\int_{B_R\cap\Omega}\mathbf{u}^\epsilon(\mathbf{y},t)\mathrm{d}\mathbf{y}\right|
&=\left|\frac{1}{|B_R\cap\Omega|}
\int_{B_R\cap\Omega}\left(\mathbf{u}^\epsilon(\mathbf{x},t)
-\mathbf{u}^\epsilon(\mathbf{y},t)\right)\mathrm{d}\mathbf{y}\right|\notag\\
&\leq\frac{1}{|B_R\cap\Omega|}C(\tau)\int_{B_R\cap\Omega}
|\mathbf{x}-\mathbf{y}|^{\frac12}\mathrm{d}\mathbf{y}\leq C(\tau)R^{\frac12}.\notag
\end{align*}
Then, for $0<\tau\leq t_1<t_2<\infty$, it follows that
\begin{align*}
|\mathbf{u}^\epsilon(\mathbf{x},t_2)-\mathbf{u}^\epsilon(\mathbf{x},t_1)|
&\leq\frac{1}{|B_R\cap\Omega|}\int_{t_{1}}^{t_{2}}
\int_{B_R\cap\Omega}|\mathbf{u}_{t}^{\epsilon}(\mathbf{y},t)
|\mathrm{d}\mathbf{y}\mathrm{d}t+C(\tau)R^{\frac12}\notag\\
&\leq CR^{-1}|t_{2}-t_{1}|^{\frac{1}{2}}\left(\int_{t_{1}}^{t_{2}}
\int\left|\mathbf{u}_{t}^{\epsilon}(\mathbf{y},t)\right|^{2}
\mathrm{d}\mathbf{y}\mathrm{d}t\right)^{\frac{1}{2}}+C(\tau)R^{\frac12}\notag\\
&\leq CR^{-1}|t_{2}-t_{1}|^{\frac{1}{2}}\left(\int_{t_{1}}^{t_{2}}
\int\left(|\dot{\mathbf{u}}^{\epsilon}|^{2}
+|\mathbf{u}^{\epsilon}|^{2}|\nabla\mathbf{u}^{\epsilon}|^{2}\right)
\mathrm{d}\mathbf{y}\mathrm{d}t\right)^{\frac{1}{2}}
+C(\tau)R^{\frac12}\notag\\
&\leq C(\tau)R^{-1}|t_{2}-t_{1}|^{\frac{1}{2}}+C(\tau)R^{\frac12},\notag
\end{align*}
owing to
\begin{align}
\int_{t_1}^{t_2}\int|\mathbf{u}^\epsilon|^2
|\nabla\mathbf{u}^\epsilon|^2\mathrm{d}\mathbf{x}\mathrm{d}t
&\leq C\sup_{t_1\leq t\leq t_2}\|\mathbf{u}^\epsilon\|_{L^\infty}^2\int_{t_1}^{t_2}
\int|\nabla\mathbf{u}^\epsilon|^2\mathrm{d}\mathbf{x}\mathrm{d}t\notag\\
&\leq C\sup_{t_1\leq t\leq t_2}\|\mathbf{u}^\epsilon\|_{L^2}^{\frac{2}{3}}
\|\nabla\mathbf{u}^\epsilon\|_{L^4}^{\frac{4}{3}}
\int_{t_1}^{t_2}\int|\nabla\mathbf{u}^\epsilon|^2\mathrm{d}\mathbf{x}\mathrm{d}t\leq C(\tau).\notag
\end{align}
Choosing $R=|t_2-t_1|^{\frac13}$, one sees that
\begin{equation}\label{4.2}
|\mathbf{u}^\epsilon(\mathbf{x},t_2)-\mathbf{u}^\epsilon(\mathbf{x},t_1)|
\leq C(\tau)|t_{2}-t_{1}|^{\frac{1}{6}},~~\text{for}~~0<\tau\leq t_1<t_2<\infty.
\end{equation}
The same estimates in \eqref{4.1} and \eqref{4.2} also hold for the magnetic filed $\mathbf{B}^\epsilon$, which implies that $\{\mathbf{u}^\epsilon\}$ and $\{\mathbf{B}^\epsilon\}$ are uniformly H\"older continuous away from $t=0$.

Hence, for any fixed $\tau$ and $T$ with $0<\tau<T<\infty$, it follows from Ascoli--Arzel\`{a} theorem that there is a subsequence $\{\epsilon_{k}\}$ with $\epsilon_k\rightarrow0$ satisfying
\begin{equation}\label{4.3}
\mathbf{u}^{\epsilon_{k}}\rightarrow \mathbf{u}~~\mathrm{and}~~
\mathbf{B}^{\epsilon_{k}}\rightarrow \mathbf{B}
~~\mathrm{uniformly}~\mathrm{on}~\mathrm{compact}~\mathrm{sets} ~\mathrm{in}~\Omega\times(0,\infty).
\end{equation}
Moreover, by the standard compactness arguments as in \cite{EF01,PL98}, we can extract a further subsequence $\epsilon_{k'}\rightarrow0$ such that
\begin{equation}\label{4.4}
  \rho^{\epsilon_{k'}}\rightarrow\rho~~\mathrm{strongly}~\mathrm{in}~L^p(\Omega),~~\mathrm{for}~
  \mathrm{any}~p\in[1,\infty)~\mathrm{and}~t\geq0.
\end{equation}
Therefore, taking the limit of $\epsilon_{k'}\rightarrow0$ in \eqref{4.3} and \eqref{4.4}, we conclude that the limit triplet $(\rho,\mathbf{u},\mathbf{B})$ is a weak solution to the problem \eqref{a1}--\eqref{a3} in the sense of Definition $\ref{d1.1}$ and satisfies \eqref{reg}.
\end{proof}

\section{Proof of Theorem \ref{t1.2}}\label{sec5}
This section is devoted to the incompressible limit of \eqref{a1}--\eqref{a3} as the bulk viscosity tends to infinity.

\begin{proof}[Proof of Theorem \ref{t1.2}.]
Let $\{(\rho^\lambda,{\bf u}^\lambda,{\bf B}^\lambda)\}$ be the family of solutions to the problem \eqref{a1}--\eqref{a3} obtained in  Theorem \ref{t1.1}. Applying \eqref{reg} and performing a similar argument as that in \eqref{4.3}, then there is a subsequence $\{(\rho^{\lambda_{k}},{\bf u}^{\lambda_{k}},{\bf B}^{\lambda_{k}})\}$ such that
\begin{align}
{\bf u}^{\lambda_{k}}\rightarrow &{\bf v},~~{\bf B}^{\lambda_{k}}\rightarrow {\bf b}
~~\text{uniformly on compact sets  in}~\Omega\times(0,\infty),\notag\\
\rho^{\lambda_{k}}&\rightarrow \varrho\ \    \text{weakly in}\ \ L^p(\Omega),\ \ \text{for any}\ p\in[1,\infty)\ \text{and}\ {t\ge 0},\label{5.1}\\
\rho^{\lambda_{k}}&\rightarrow \varrho\ \  \text{weakly* in}\ \ L^\infty(\Omega),\ \ \text{for any}\ {t\ge 0},\notag\\
\divv{\bf u}^{\lambda_{k}}&\rightarrow 0~~\text{strongly  in}~L^2(\Omega\times(0,\infty)).\notag
\end{align}
Hence, we conclude that $\divv{\bf v}=0$ and $(\varrho,{\bf v},{\bf b})$ satisfies \eqref{1.13}--\eqref{1.15} for all $C^1$ test functions $(\phi,\boldsymbol\psi)$ just as in Definition \ref{d1.2}, with $\divv\boldsymbol\psi=0$ on $\Omega\times[0,\infty)$. Moreover, $(\varrho,{\bf v},{\bf b})$ satisfies
\begin{equation}\label{5.2}
 0\leq\varrho({\bf x},t)\leq 2 \hat\rho\ \ \text{a.e. on} \
\Omega\times[0,\infty),
\end{equation}
\begin{equation}\label{5.3}
\begin{aligned}[b]
&\sup\limits_{t\ge 0}\big(\|\sqrt{\varrho}{\bf v}\|_{L^2}^2+\|{\bf v}\|_{H^1}^2+\sigma\|\nabla^2{\bf v}\|_{L^2}^2+\|{\bf b}\|_{H^1}^2+\sigma\|\nabla\curl{\bf b}\|_{L^2}^2\big)\\
&\quad
+\int_0^\infty\big(\mu\|{\bf v}\|_{H^1}^2+\|\nabla^2{\bf v}\|_{L^2}^2
+\nu\|{\bf b}\|_{H^1}^2+\|\nabla\curl{\bf b}\|_{L^2}^2
\big)\mathrm{d}\tau\le C(C_0,M).
\end{aligned}
\end{equation}

It remains to show \eqref{1.11} and \eqref{1.12}. It follows from the mass equation $\eqref{a1}_1$ that
\begin{equation*}
\partial_t(\rho^{\epsilon,\lambda}-\rho_0^\epsilon)^2+{\bf u}^{\epsilon,\lambda}\cdot\nabla(\rho^{\epsilon,\lambda}-\rho_0^\epsilon)^2
+2\rho^{\epsilon,\lambda}(\rho^{\epsilon,\lambda}-\rho_0^\epsilon)\divv{\bf u}^{\epsilon,\lambda}=0.
\end{equation*}
Integrating the above equality over $\Omega\times(0,t)$, one has that
\begin{align*}
\big\|(\rho^{\epsilon,\lambda}-\rho_0^\epsilon)(\cdot,t)\big\|_{L^2}^2
& =\int_0^t\int(\rho^{\epsilon,\lambda}-\rho_0^\epsilon)^2\divv{\bf u}^{\epsilon,\lambda}\mathrm{d}{\bf x}\mathrm{d}\tau-2\int_0^t\int\rho^{\epsilon,\lambda}(\rho^{\epsilon,\lambda}
-\rho_0^\epsilon)\divv{\bf u}^{\epsilon,\lambda}\mathrm{d}{\bf x}\mathrm{d}\tau
\\ & \le C\bigg(\int_0^t\big\|\rho^{\epsilon,\lambda}-\rho_0^\epsilon\big\|_{L^4}^4
\mathrm{d}\tau\bigg)^\frac{1}{2}\bigg(\int_0^t\big\|\divv{\bf u}^{\epsilon,\lambda}\big\|_{L^2}^2\mathrm{d}\tau\bigg)^\frac{1}{2}
\\ & \quad +C\sup\limits_{t\ge 0}\big\|\rho^{\epsilon,\lambda}(\cdot,t)\big\|_{L^\infty} \bigg(\int_0^t\big\|\rho^{\epsilon,\lambda}-\rho_0^\epsilon\big\|_{L^2}^2
\mathrm{d}\tau\bigg)^\frac{1}{2}\bigg(\int_0^t\big\|\divv{\bf u}^{\epsilon,\lambda}\big\|_{L^2}^2\mathrm{d}\tau\bigg)^\frac{1}{2} \\
& \le C(t)\lambda^{-\frac{1}{2}},
\end{align*}
which along with \eqref{4.4} yields that
\begin{equation*}
\|(\rho^\lambda-\rho_0)(\cdot,t)\|_{L^2}^2
=\lim\limits_{\epsilon_k\rightarrow0}
\big\|(\rho^{\epsilon_k,\lambda}-\rho_0^{\epsilon_k})(\cdot,t)\big\|_{L^2}^2
\le C(t)\lambda^{-\frac{1}{2}}.
\end{equation*}
Thus
\begin{equation}\label{5.4}
\lim\limits_{\lambda\rightarrow\infty}\|(\rho^\lambda-\rho_0)(\cdot,t)\|_{L^2}
=0,\ \ \text{for any}\ t\ge 0.
\end{equation}

Next, using the mollifier $j_\epsilon$ as test functions in \eqref{1.13}, one infers from \eqref{5.3} that, for any compact set $K\subset\Omega$,
\begin{equation}\nonumber\partial_t[\varrho]_\epsilon+{\bf v}\cdot\nabla [\varrho]_\epsilon=\divv\left([\varrho]_\epsilon {\bf v}\right)-\divv[\rho{\bf v}]_\epsilon~~ \text{a.e. on}\ K\times(0,\infty),
\end{equation}
and furthermore,
\begin{equation*}
\partial_t\left([\varrho]_\epsilon-\rho_0\right)^2
+{\bf v}\cdot\nabla \left([\varrho]_\epsilon-\rho_0\right)^2=2\left([\varrho]_\epsilon-\rho_0\right)\big(\divv\left([\varrho-\rho_0]_\epsilon {\bf v}\right)-\divv\left[(\varrho-\rho_0){\bf v}\right]_\epsilon\big)~~ \text{a.e. on}~  K\times(0,\infty).
\end{equation*}
Integrating the above equality over $K\times(0,t)$, one gets that
\begin{align}\label{5.5}
&\big|\|([\varrho]_{\epsilon}-\rho_0)(\cdot,t)\|_{L^2(K)}^2- \|[\rho_0]_{\epsilon}-\rho_0\|_{L^2(K)}^2\big|\notag\\
 & \le C\sup\limits_{t\ge 0}\|[\varrho]_\epsilon-\rho_0\|_{L^\infty}\int_0^t
 \|\divv\left([\varrho-\rho_0]_\epsilon {\bf v}\right)-\divv\left[(\varrho-\rho_0){\bf v}\right]_\epsilon\|_{L^1(K)}\mathrm{d}\tau
\notag \\
& \le C\int_0^t\|\divv\left([\varrho-\rho_0]_\epsilon {\bf v}\right)-\divv\left[(\varrho-\rho_0){\bf v}\right]_\epsilon\|_{L^1(K)}\mathrm{d}\tau.
\end{align}
Lemma \ref{lcom} implies that
\begin{equation*}
\|\divv\left([\varrho-\rho_0]_\epsilon {\bf v}\right)-\divv\left[(\varrho-\rho_0){\bf v}\right]_\epsilon\|_{L^1(K)}\le C(K)\|\varrho-\rho_0\|_{L^2(\Omega)}\|{\bf v}\|_{W^{1,2}(\Omega)}\in L^1(0,T),\ \ \text{for any}\ T>0.
\end{equation*}
This together with Lebesgue's dominated convergence theorem and Lemma \ref{lcom} leads to
\begin{align}\label{5.6}
&\lim\limits_{\epsilon\rightarrow0}\int_0^t
\|\divv\left([\varrho-\rho_0]_\epsilon {\bf v}\right)-\divv\left[(\varrho-\rho_0){\bf v}\right]_\epsilon\|_{L^1(K)}\mathrm{d}\tau \notag \\
& =\int_0^t\lim\limits_{\epsilon\rightarrow0}\|\divv\left([\varrho-\rho_0]_\epsilon {\bf v}\right)-\divv\left[(\varrho-\rho_0){\bf v}\right]_\epsilon\|_{L^1(K)}\mathrm{d}\tau=0.
\end{align}
Substituting \eqref{5.6} into \eqref{5.5}, we conclude that
\begin{align}\label{5.7}
\|(\varrho-\rho_0)(\cdot,t)\|_{L^2(K)}^2= \lim\limits_{\epsilon\rightarrow0} \big|\| ([\varrho]_{\epsilon}-\rho_0)(\cdot,t)\|_{L^2(K)}^2- \|[\rho_0]_{\epsilon}-\rho_0\|_{L^2(K)}^2\big|=0,
\end{align}
which implies \eqref{1.12}.

Finally, combining \eqref{5.4} and \eqref{5.7}, it follows that
\begin{align*}
\lim\limits_{\lambda\rightarrow\infty}\|(\rho^\lambda-\varrho)(\cdot,t)\|_{L^2(K)}
=0,\ \ \text{for any compact set} \ K \subset \Omega \text{ and any} \ t\ge 0,
\end{align*}
which along with \eqref{5.1} yields \eqref{1.11}. Consequently, $(\varrho,{\bf v},{\bf b})$ is a global weak solution to the inhomogeneous incompressible magnetohydrodynamic equations \eqref{1.5} in the sense of Definition \ref{d1.2}.
\end{proof}

\section*{Conflict of interests}
The authors declare that they have no conflict of interests.

\section*{Data availability}
No data was used for the research described in the article.

\end{document}